\documentclass[10pt, reqno]{amsart}
\usepackage{a4wide}
\usepackage[english]{babel}
\usepackage{graphicx}
\usepackage{amsmath}
\usepackage{amssymb}
\usepackage{bbm}
\usepackage{latexsym}
\usepackage{eurosym}     
\usepackage{caption}
\usepackage{subcaption}
\usepackage{amsaddr}

\usepackage{fancyhdr}
\usepackage{parskip}
\usepackage{lscape}
\usepackage[ruled,vlined]{algorithm2e}
\usepackage{enumitem}
\usepackage[final]{pdfpages}
\usepackage{array}
\usepackage[utf8]{inputenc}
\usepackage{xspace}                 
\usepackage{amsthm}
\usepackage{hyperref}
\newtheorem{theorem}{Theorem}[section]
\newtheorem{corollary}[theorem]{Corollary}
\newtheorem{lemma}[theorem]{Lemma}
\newtheorem{prop}[theorem]{Proposition}
\newtheorem{defn}[theorem]{Definition}
\newtheorem{remark}[theorem]{Remark}
\newtheorem{claim}[theorem]{Claim}
\newtheorem{cond}[theorem]{Condition}

\newcommand{\Ind}[1]{\mathbbm{1}\{#1\}}

\newcommand{\cP}{\mathcal{P}}

\newcommand{\cL}{\mathcal{L}}
\newcommand{\cT}{\mathcal{T}}

\newcommand{\wT}{\widetilde{T}}

\DeclareMathOperator*{\argmin}{arg\,min}

\newcommand{\uc}{\bar{c}}
\newcommand{\lc}{\underline{c}}

\newcommand{\Tmd}{\mathrm{Tmd}}
\newcommand{\tr}{\mathrm{tr}}
\newcommand{\typ}{\mathrm{typ}}
\newcommand{\hgt}{\mathrm{hgt}}

\begin{document}
\fancyhead[LO]{Threshold metric dimension of trees}
\fancyhead[LE]{Bartha, Komj{\'a}thy and Raes}
\fancyhead[RO]{}
\fancyhead[RE]{}
\title{Sharp bound on the threshold metric dimension of trees}
\date{\today}
\author{Zsolt Bartha$^\star$, J\'ulia Komj\'athy$^\dagger$ and J\"arvi Raes$^\ddag$}
\address{$^\star$ Department of Mathematics and Computer Science, Eindhoven University of Technology\\
$^\dagger$ Delft Institute of Applied Mathematics, Delft University of Technology\\
$^\ddag$ Department of Mathematics and Computer Science, Eindhoven University of Technology}
\email{z.bartha@tue.nl, J.Komjathy@tudelft.nl, j.raes@student.tue.nl}

\maketitle
\begin{abstract}
The threshold-$k$ metric dimension ($\Tmd_k$) of a graph is the minimum number of sensors --  a subset of the vertex set -- needed to uniquely identify any vertex in the graph, solely based on its distances from the sensors, when the measuring radius of a sensor is $k$. We give a sharp lower bound on the $\Tmd_k$ of trees, depending only on the number of vertices $n$ and the measuring radius $k$. This sharp lower bound grows linearly in $n$ with leading coefficient  $3/(k^2+4k+3+\mathbbm{1}\{k\equiv 1\pmod 3\})$, disproving earlier conjectures by Tillquist \emph{et al.} in \cite{tfl2021tmd} that suspected $n/(\lfloor k^2/4\rfloor +2k)$ as main order term. 
We provide a construction for the largest possible trees with a given $\Tmd_k$ value.  The proof that our optimal construction cannot be improved  relies on edge-rewiring procedures of arbitrary (suboptimal) trees with arbitrary  resolving sets, which reveal the \emph{structure} of how small subsets of sensors measure and resolve certain areas in the tree that we call the attraction of those sensors. The notion of `attraction of sensors' might be useful in other contexts beyond trees to solve related problems.
We also provide an improved lower bound on the $\Tmd_k$ of arbitrary trees that takes into account the structural properties of the tree, in particular, the number and length of simple paths of degree-two vertices terminating in leaf vertices. This bound complements \cite{tfl2021tmd}, where only trees \emph{without} degree-two vertices were considered, except the simple case of a single path. 
\end{abstract}

\section{Introduction}

The \emph{metric dimension} of graphs is a combinatorial notion first introduced by Slater \cite{slater1975leaves} in 1975, and independently by Harary and Melter \cite{harary1976metric} one year later. It is the optimal value of a source detection problem described as follows. Let $G=(V,E)$ be a simple, undirected graph, and let $d$ denote the graph distance on its vertex set $V$, with the convention that $d(x,x)=0$. We call a subset of vertices $S\subseteq V$ \emph{resolving}, if the vector of graph distances $(d(s,v))_{s\in S}$ is distinct for each vertex $v\in V$. In other words, we imagine that the vertices in $S$ are \emph{sensors} that can measure their distances from each vertex in the graph, and call $S$ a resolving set if each vertex in $V$ can be (uniquely) identified by the measurements of the sensors. Then the metric dimension of $G$ is the smallest cardinality of such a resolving set. This model of source detection is motivated by the problem of finding the unknown source of an infection in a network, based on the measured infection time of a certain subset of individuals, as in \cite{pinto2012locating,spinelli2016observer}.

In this paper we consider a modified version of the above problem, where we replace the graph distance $d$ above by its truncated form $d_k(\cdot,\cdot)=\min\{d(\cdot,\cdot),k+1\}$, with $k$ being an integer parameter. This corresponds to limiting the radius of measurement of each sensor vertex to $k$, i.e., not allowing the sensor to distinguish between vertices that are further away than $k$. A set $S\subseteq V$ is called a \emph{threshold-$k$ resolving set} if $S$ is a resolving set under the metric induced by $d_k(\cdot, \cdot)$, and additionally for every $v\in V$ there is some $s\in S$ with $d_k(s,v)\le k$. The smallest cardinality of such a set is the \emph{threshold-$k$ metric dimension} of $G$, denoted by $\Tmd_k(G)$. This modification of the model is inspired by the scenario where the sensors can only make noisy measurements: the noise accumulates over distance, and above a certain threshold the measurements become unreliable, as in \cite{pinto2012locating, spinelli2016observer, lecomte2021adaptivity}. Works where the threshold-$k$ metric dimension has been introduced include \cite{estrada2021k, tillquist2020low,tfl2021tmd} and \cite{geneson2021distance}, and for $k=1$ the threshold metric dimension is equivalent to the locating-dominating code problem, introduced by Slater earlier in \cite{slater1988dominating}. We discuss further related literature, including variations on the metric dimension problem below in Section \ref{sec:related-work}. 

In this paper we study the threshold metric dimension of \emph{trees}  (cycle-free connected graphs), see Section \ref{sec:model} below for the usual definition. Our main result is a worst-case lower bound on $\Tmd_k(G)$ based on the number of vertices in the tree.

\begin{theorem}[Lower bound on the threshold-$k$ metric dimension]\label{thm:main}
Let $T$ be any tree on $n\ge 1$ vertices. Then for all $k\ge 1$,
\[\Tmd_k(T)\ge\left\lceil\frac{3n+k^2+k+\mathbbm{1}_{k\equiv 1\!\!\pmod 3}}{k^2+4k+3+\mathbbm{1}_{k\equiv 1\!\!\pmod 3}}\right\rceil.\]
\end{theorem}

These bounds are sharp, in the sense that for any positive integers $n$ and $k$ there exists a tree $T$ on $n$ vertices, which satisfies the respective bound. We prove Theorem \ref{thm:main} by identifying the size of the largest tree with a given $\Tmd_k$ (see Proposition \ref{prop:main} below). That result disproves the conjecture of a recent paper by Tillquist, Frongillo and Lladser \cite{tfl2021tmd}, which speculated that the size of the largest tree $T$ with $\Tmd_k(T)=m$ is $\Theta(mk^2/4)$, based on their construction. Our result shows that it is, in fact, $\Theta(mk^2/3)$. We also provide a construction of trees of optimal size, and from our proof it follows that the optimal tree is non-unique. In fact, for each $m$, the number of largest-size (i.e., optimal) trees that can be resolved by $m$ sensors is at least as large as the number of non-isomorphic unlabelled trees on $m$ vertices. 

The largest part of our paper is devoted to the proof that no tree on $n$ vertices can be measured by less sensors than the lower bound in Theorem \ref{thm:main}. In principle, there could be two proof strategies to show such a lower bound. Either one `spares' a sensor on any suboptimal tree, i.e., one shows that the tree can be resolved by less sensors if it does not follow the optimal construction. This strategy however, does not work, since it is not hard to construct suboptimal trees and resolving sets where one cannot spare a sensor: such an example is a star-graph (a central vertex connected to $n-1$ leafs). The star-graph needs at least $n-2$ sensors for all $k$ and it is not hard to show that no sensor can be removed. 

The second possible strategy to show that a given labelled tree is suboptimal is to keep the sensors in place and add a new vertex to the tree, while still ensuring that every vertex is uniquely resolved. We follow this latter proof strategy.  We add a vertex via a series of `transformations', which can be applied to any tree \emph{not} following the optimal construction. These transformations all preserve $\Tmd_k$ and do not decrease the number of vertices.
In more detail:  given a labelled tree with a resolving set that does not follow the optimal construction, we slightly modify the edge set by rewiring a few edges and possibly adding a few labelled vertices, and show that the obtained (possibly) larger tree is still measured by the same sensor vertices, which violates the assumption that the tree was largest possible. 
Since the optimal tree is non-unique, these transformations either do not change the number of vertices and result in an optimal construction that we describe, or else, when we could add a vertex, they result in a tree $T'$ on a strictly larger vertex set with $\Tmd_k(T')=m$. 

Some notions that we introduce during the proofs, especially what we call \emph{attraction of sensors},  might be useful in other contexts as well, because they uncover the structure of the vertices measured by a subset of sensors and as a result they reveal where optimality may be violated.

Besides giving a worst-case lower bound on $\Tmd_k(T)$ in Theorem \ref{thm:main}, we also provide a sharper lower bound for certain suboptimal trees, which takes into account the \emph{structural properties} of the particular tree in question (see Theorem \ref{thm:LPLowerBound}). This bound is obtained by providing a locally optimal placement of sensors around certain degree-two paths terminating in leaves of the tree, and then applying Theorem \ref{thm:main} for the rest of the graph. In comparison to \cite{tfl2021tmd}, which identifies the threshold-$k$ metric dimension of a certain subclass of trees \emph{without} degree-two vertices, our lower bound builds on exactly these degree-two structures. While our lower bound might be suboptimal on most trees, there are trees (with leaf-paths) on which it provides sharp bounds, hence in some sense the bound cannot be improved, at least not in full generality.

\subsection{Related work and open questions}\label{sec:related-work}

\emph{Algorithmic aspects.} The question of finding the metric dimension of graphs has been extensively studied from an algorithmic point of view. The problem on general graphs is NP-hard \cite{khuller1996landmarks}, and can only be approximated up to a factor $\log n$  \cite{beerliova2006network,hauptmann2012approximation}. For parameter values $k\ge 2$,  threshold-$k$ metric dimension of general graphs is also an NP-hard problem \cite{estrada2021k,fernau2018adjacency}.
For trees, on the other hand, \cite{khuller1996landmarks} provides a simple linear-time algorithm for the computation of the metric dimension, writing it as the difference between the number of leaves and the number of vertices that have degree at least $3$ and are the endpoints of at least one simple path in the tree (which we will call a \emph{leaf-path}). This idea is closely related to our approach to the improved lower bound in Theorem \ref{thm:LPLowerBound}. As for the threshold-$k$ metric dimension for $k=1$, known as the location-domination number, the early work of Slater \cite{slater1987domination,slater1988dominating} shows that the location-domination number can be computed in linear time on trees, and gives a lower bound on its value, which was later improved by \cite{blidia2007locating}, and further improved by \cite{slater2014sharp}. We are unaware of such a linear-time algorithm for the threshold-$k$ metric dimension on trees for general  $k\ge2$.

\emph{Graph theoretical aspects.} Many aspects of the metric dimension have been studied from the graph theoretic point of view as well, including bounds in terms of the diameter of the graph \cite{chartrand2000resolvability, hernando2010diameter}, bounds for certain Cayley graphs \cite{fehr2006cayley}, Cartesian products of graphs \cite{caeres2007cartesian}, and the metric dimension of infinite graphs \cite{caceres2012infinite}, and wheel graphs \cite{buczkowski2003k, siddiqui2014computing}.

\emph{Asymptotic metric dimension of random graph classes.} There has been much work studying the metric dimension of \emph{random graphs}. Its asymptotic value for dense Erd\H{o}s--R\'enyi graphs $G(n,p)$ is obtained in \cite{bollobas2013ER}, while for sparse, subcritical $G(n,p)$ and uniform random trees and forests, its asymptotic distribution is shown in \cite{dieter2015clt}. The asymptotic minimal size of an \emph{identifying set} for the dense $G(n,p)$ model is also known \cite{frieze2007ER}. This latter is similar to a locating-dominating set (threshold-1 resolving set) with the difference that sensors cannot distinguish their own vertex from their immediate neighbors. Both this problem and the location-domination number are studied in \cite{muller2009geometric} for random geometric graphs, where an identifying set only exists with an asymptotically nontrivial probability in certain ranges of the parameters, unlike in the Erd\H{o}s--R\'enyi setting. For various models of growing tree models (e.g.\ the random binary search tree, preferential attachment trees, etc), and conditioned Galton-Watson trees (including uniform trees on $n$ vertices) Law of Larger Numbers-type results for the metric dimension were obtained by \cite{komjathy2020metric}.

\emph{Modifications of the metric dimension.} We also mention a few results on some modified versions of the metric dimension problem. In \cite{thiran2018doublemd} bounds are given for the \emph{double metric dimension} of the Erd{\H o}s-R\'enyi random graph model. This version of the problem assumes that the infection time of the source (that is, the starting time of the process) is unknown, and requires a set of sensors $S$ that is \emph{double resolving}: every vertex $v$ is uniquely identified by the vector of distance differences $\{d(s_1,v)-d(s_2,v):s_1,s_2\in S\}$. 
Another variant, the \emph{sequential metric dimension} (SMD) of $G(n,p)$ graphs is studied in \cite{odor2019seq}: in this model one is allowed to determine the placement of sensors in an adaptive fashion, using the measurements of the previously deployed ones, to determine the location of the source. The SMD of a graph is then the minimal number of sensors needed in a worst-case position of the source. A more general version of this problem is the $k$-metric dimension, where every pair of vertices needs to be resolved by at least $k$ sensors (and where this notation $k$ is used differently from our paper), has been studied in \cite{estrada2016adjacency, estrada2015k, estrada2014k, estrada2016k, estrada2016kl}.
A further variant of this circle of problems is the \emph{fractional metric dimension}, introduced in \cite{currie2001metric}, which is the linear programming relaxation of the integer programming problem encoding the identification of the metric dimension of a graph. Further work on this concept include \cite{arumugam2012fractional, arumugam2013fractional, feng2013metric, feng2014fractional}.

Slightly less related to our work are the concepts of $r$-\emph{identifying codes} and $r$-\emph{locating-dominating codes}, where a sensor can measure up to distance $r$ but cannot distinguish between vertices within this radius (in the case of the identifying code), except itself (in the case of $r$-locating-dominating codes). The authors of \cite{chen2011identifying} provide the sizes of the optimal $r$-identifying codes of paths and cycles, which asymptotically coincide with the optimal size of a $1$-identifying code, i.e., it has density $1/2$. Similar results for the $r$-locating-dominating code for cycles in \cite{exoo2011locating} show that the optimal code in this case also has the same asymptotic density, $1/3$, as in the $r=1$ case. For other work on this topic, see the references within \cite{chen2011identifying}.

\emph{Open questions.}
Next, we mention questions that our paper leaves open.  The question of (algorithmically) quickly finding an optimal arrangement of sensors stays open on trees for general $k\ge 2$. Our proof of Theorem \ref{thm:LPLowerBound} gives partial answer by finding the optimal placement on certain leaf-paths, parts of the tree that are simple paths leading to a leaf vertex. For the rest of the tree, however, only the size-dependent lower bound in Theorem \ref{thm:main} is used.

Another possible direction is to study the threshold metric dimension of random graph models, starting with e.g. random trees, as in  \cite{komjathy2020metric}. While Law of Large Numbers seems to be reasonable to hold for the tree-classes studied in \cite{komjathy2020metric}, at the moment we do not see an easy way to prove this.

The rest of the paper is structured as follows. In Section \ref{sec:model} we state our results after providing the necessary notation and definitions, and give a short sketch of our proofs. In Section \ref{sec:prelim} we introduce further definitions and concepts that will be used throughout the paper. Sections \ref{sec:A}, \ref{sec:B} and \ref{sec:C} contain the three transformations that form the three main steps of the proof of Theorem \ref{thm:main}. In Section \ref{sec:OptTree} we complete the proof of Theorem \ref{thm:main}, while also giving a construction for trees of optimal size. Finally, in Section \ref{sec:general} we prove the structural lower bound, Theorem \ref{thm:LPLowerBound}.

\section{Model and results} \label{sec:model}
We will start by fixing the notation that we will use throughout the paper. For a set $V$ let $V^{(2)}$ denote the 2-element subsets of $V$. A (simple, undirected) {\it graph} $G$ is an ordered pair $(V,E)$ where $V$ is a set of {\it vertices} and $E\subseteq V^{(2)}$ is a set of {\it edges}. We say that $e\in E$ {\it connects} its two vertices in the graph. We will sometimes write $V(G)$ and $E(G)$ for the vertex and edge sets, respectively, of a given graph $G$. Somewhat abusing the notation we will sometimes write $v\in G$ instead of $v\in V(G)$ for a vertex of $G$, as is standard. The {\it size} of a graph $G$, denoted by $|G|$, is the cardinality of its vertex set. A {\it subgraph} of $G$ is a pair $(V',E')$ such that $V'\subseteq V$ and $E'\subseteq E\cap V'^{(2)}$. A subgraph {\it induced} by $V_1\subseteq V$ in $G$ is the pair $(V_1,E_1)$ where $E_1=E\cap V_1^{(2)}$. A {\it path} is a graph $P$ such that $V(P)=\{v_i\}_{i=1}^{k+1}$, for some $k\ge 1$, and $E(P)=\{\{v_i,v_{i+1}\}\}_{i=1}^k$. We call $v_1, v_{k+1} $ the end vertices of the path. A {\it cycle} is a graph $C$ such that $V(C)=\{v_i\}_{i=1}^{k}$, $k\ge 3$, and $E(C)=\cup_{i=1}^{k-1}\{\{v_i,v_{i+1}\}\}\cup\{\{v_k,v_1\}\}$. The {\it length} of a path $P$ or a cycle $C$ is the number of edges in it. A graph $G$ is {\it connected} if for any pair $u,v\in V(G)$ there is a path in $G$ (as a subgraph) that contains both $u$ and $v$. We call a graph a {\it forest} if it has no cycles in it, and we call a connected forest a {\it tree}.  The degree of a vertex $v$, denoted by $\deg(v)$, is the number of edges containing $v$. A vertex of degree one is called a {\it leaf}. We call a path $P$ a {\it leaf-path} in $G$ if $P$ is the induced subgraph of $G$ on the vertices $V(P)$, and, with respect to $G$, one end vertex of $P$ is a leaf, the other has degree at least 3, and all its other vertices have degree 2. The {\it distance} between two vertices $u,v$ in $G$, denoted by $d_G(u,v)$, is the length of the shortest path between $u$ and $v$ in $G$, with the convention that $d_G(u,u)=0$. Here we omit the subscript when the underlying graph is clear from the context. Finally, we denote by $k\wedge l=\min\{k,l\}$, and $d_k(x,y):=d(x,y)\wedge (k+1)$.

Next we define the main topic of this paper.

\begin{defn}\label{def:tmd}
Let $G = (V,E)$ be an arbitrary simple, undirected graph, and fix an integer threshold $k\ge 1$.  We say that a vertex $s$ \emph{resolves} (or equivalently, \emph{distinguishes}) a pair of vertices $x,y\in V$ if $d_k(s,x) \neq d_k(s,y)$.
A \emph{threshold-$k$ resolving set} is a subset $S$ of $V$ such that for every pair of vertices $x,y\in V$ there is some vertex $s\in S$ such that $s$ resolves $x,y$, and for every $x\in V$ there is some $s\in S$ such that $d_k(s,x)\le k$. The threshold-$k$ metric dimension of $G$, denoted by $\Tmd_k(G)$, is the smallest integer $m$ such that there exists a threshold-$k$ resolving set $S$ for $G$ with $|S| = m$.
\end{defn}
We call the elements of a threshold-$k$ resolving set \emph{sensors}. We say that a vertex $x$ is \emph{measured by} a sensor $s$ if $d_k(s, x) \le k$.

Our main result, Theorem \ref{thm:main}, readily implies the result of Slater \cite{slater1987domination} about the threshold-$1$ metric dimension, which is identical to the locating-dominating number of the tree.

\begin{corollary}[Lower bound on the locating-dominating number \cite{slater1987domination}]
Let $T$ be any tree on $n$ vertices. Then
\[\Tmd_1(T)\ge\left\lceil\frac{n+1}{3}\right\rceil.\]
\end{corollary}

To prove Theorem \ref{thm:main} we will identify the largest trees with a given threshold-$k$ metric dimension. To state that result we introduce some notation first.

\begin{defn}Fix $k\ge 1$. We denote the set of trees with threshold-$k$ metric dimension $m$ by $\cT_m=\cT_m(k)$, and $\cT^\star_m=\cT_m^\star(k)$ denotes the set of trees with the largest possible size within $\cT_m$:
\begin{align*}
    \cT_m&:=\{T\text{ tree}:\Tmd_k(T)=m\},\\
    \cT_m^\star&:=\{T^
\star\in\cT_m: |T^\star|=\max_{T\in\cT_m}(|T|)\}.
\end{align*}
\end{defn}
We then identify the maximal size of a tree that can be measured by $m$ sensors.
\begin{prop}\label{prop:main}
For all $k\ge 1$, and any $T^\star\in \mathcal T_m^\star$,
\begin{equation}
\begin{aligned}
     |T^\star|&=m\cdot\frac{k^2+4k+3+\mathbbm{1}_{k \equiv 1\ (\mathrm{mod}\  3) }}{3} - \frac{k^2+k+\mathbbm{1}_{k \equiv 1\ (\mathrm{mod}\  3) }}{3}  \\ &=(k+1)m+(m-1)(k^2+k+\mathbbm{1}_{k \equiv 1\ (\mathrm{mod}\  3) })/3.
\end{aligned}
   \end{equation}
\end{prop}

\begin{remark}
Recall from Definition \ref{def:tmd} that we require for a resolving set $S$ that every vertex $v\in V$ be measured by at least one sensor in $S$. We use this convention to make the presentation of our results slightly more convenient.
Omitting this requirement would simply add one extra vertex to the optimally-sized trees in $\cT^\star_m$ that is not measured by any sensor, and would change the bounds in Theorem \ref{thm:main} accordingly. 
\end{remark}

Next, we give an improved structural lower bound on the threshold metric dimension of suboptimal trees as well. The idea is that having multiple leaf-paths emanating from a common vertex  $v$ is very costly in terms of how many sensors are needed to identify them. Since sensors that are not part of such leaf-paths can only measure such paths via $v$, they cannot distinguish vertices at equal distance from $v$ located on two different leaf-paths. Thus, we can compute how many sensors such a system of leaf-paths minimally requires. For the 'rest' of the tree, we then essentially use the optimal bound that we developed in Theorem \ref{thm:main}. Our lower bound is valid for any tree, but gives fairly sharp lower bounds only for trees that have relatively large number of leaf-paths.  To be able to state the result, we start with some definitions.

\begin{defn}[Leaf paths and support vertices]\label{def:support}We will write $\cL_v=\{\cP^{(v)}_j\}$ for the collection of leaf-paths starting at a vertex $v$ with $\deg v\ge 3$, and denote their number by $L_v=|\cL_v|$. The length (number of edges) of a leaf-path $\cP^{(v)}_j$ will be denoted by $\ell(\cP^{(v)}_j)=\ell(v,j)$. 
Define $F_T=\{v\in T: \mathrm{deg}(v)\ge 3, L_v\ge 2\}$
to be the set of {\textit support vertices} of $T$.
\end{defn}

\begin{defn}\label{def:complexity}
Fix $k\ge 1$. For an integer $\ell\ge 1$, let $q$ and $r$ be the non-negative integers such that $\ell=q(3k+2)+r$, and $r\le 3k+1$.  Define the {\textit upper} and {\textit lower complexity}, respectively, of a path of length $\ell$ to be
\begin{align*}
\uc(\ell)&:=2q+\Ind{r\ge 1}+\Ind{r\ge 2k+2},\quad\text{and}\\
\lc(\ell)&:=2q+\Ind{r\ge k+1}+\Ind{r\ge 2k+2}.
\end{align*}
\end{defn}
The next lemma identifies how many sensors a system of leaf-paths minimally requires. We place the lemma here since it might be useful also for algorithmic aspects. In the proof, below in Section \ref{sec:general}, we also provide the location of the sensors mentioned in the lemma.
\begin{lemma}\label{lem:LeafPathSensors}
If $S$ is a threshold-$k$ resolving set on $T$, and $v\in F_T$, then all but at most one of the vertex sets $V(P^{(v)}_j)\setminus\{v\}$ for $P^{(v)}_j\in\cL_v$ contain at least $\uc(\ell(v,j))$ sensors in $S$, while $V(P^\star)\setminus\{v\}$ for the remaining path $P^\star\in\cL_v$ contains at least $\lc(\ell(P^\star))$ sensors in $S$.
\end{lemma}

To determine which path shall be the special path $P^\star$ in Lemma \ref{lem:LeafPathSensors},  for a lower bound we subtract the difference between the upper and the lower complexity for each path $P_j^{(v)}$ (since this is the number of sensors `spared' by choosing $P_j^{(v)}$ to be the special path $P^\star$) and maximise it over paths in $\mathcal L_v$. 

Then the minimal number of sensors that need to be placed on $\cup_{j\le L_v} V(P_j^{(v)})\setminus\{v\}$ for some $v\in F_T$ is at least
\begin{equation}\label{eq:LeafPathSensors}
R(\cL_v)=\sum_{j=1}^{L_v}\uc(\ell(v,j))-\max_{1\le j\le L_v}\bigg\{\uc(\ell(v,j))-\lc(\ell(v,j))\bigg\}.
\end{equation}

As a combination of Theorem \ref{thm:main} and Lemma \ref{lem:LeafPathSensors}, and \eqref{eq:LeafPathSensors} we get the general lower bound on the threshold-$k$ metric dimension of trees:

\begin{theorem}\label{thm:LPLowerBound}
Let $T$ be a tree with $n$ vertices and fix $k\ge1$. Then
\begin{equation}\label{eq:tmdLeafPaths1}
\Tmd_k(T)\ge
\left\lceil\frac{3n-3\sum_{v\in F_T}\sum_{j=1}^{L_v}\ell(v,j)+k^2+k+\mathbbm{1}_{k \equiv 1\ (\mathrm{mod}\  3) }}{k^2+4k+3+\mathbbm{1}_{k \equiv 1\ (\mathrm{mod}\  3) }}\right\rceil+\sum_{v\in F_T}R(\cL_v)-|F_T|.
\end{equation}
\end{theorem}
Observe that the sum $\sum_{v\in F_T}\sum_{j=1}^{L_v}\ell(v,j)$ is the total number of vertices that are on leaf-paths emanating from support vertices.
We provide the proof in Section \ref{sec:general}.
\subsection{Discussion and methodology}
We can contrast Theorem \ref{thm:main} and Proposition \ref{prop:main} to the results of \cite{tfl2021tmd}. The authors of that paper conjecture that the leading coefficient of $m$ in Proposition \ref{prop:main} should be  $c_k=\lfloor k^2/4 \rfloor$, while we find that it is, in fact, $(k^2+4k+4)/3$ when $k\equiv 1\pmod 3$ and $(k^2+4k+3)/3$ otherwise, which is higher. The underestimation of the size of the optimal tree comes from assuming that the distance between two neighboring sensors in the tree is  exactly $k$ for all $k$, when in fact this is a parameter that can be optimised and is the nearest integer to $(2k+1)/3$.
\subsubsection*{Methodology, sketch of proof}
To  find the largest trees that have threshold-$k$ metric dimension $m$, we first find the 'skeleton' of such a tree, incorporating certain properties that certain optimal trees must satisfy. The edges of trees that do not follow these properties can be rewired in certain ways such that the threshold-$k$ metric dimension does not increase, while the number of vertices increases, or stays the same. Our proof will follow four steps that we explain next.

Step A: If the sensors on the tree $T$ are placed so that some vertices are measured only by a single sensor but are not forming a leaf-path emanating from this sensor (path of degree-2 vertices terminating in a leaf), the first transformation (Transformation A) moves these vertices into such a leaf-path. 
Once/whenever no such sensor can be found, we are able to execute two other transformations.

Step B: We call a pair of sensors neighboring if no other sensor can be found on the shortest path between them. If we can find a neighboring sensor-pair so that the shortest path between them contains more than $k+1$ edges, we can rewire the edges in a way that the distance between the two sensors becomes shorter (Transformation B), and they still remain a neighboring sensor-pair. 
A repetitive application of Transformation B results in a tree where the graph-distance between these two sensors is at most $k+1$, and where we can execute the transformation.

Step C: Finally, when we find three sensors such that two of the shortest paths between them are of length at most $k+1$, and they have a nontrivial overlap, then we apply a third transformation (Transformation C) that again rewires edges and adds one more vertex to the tree. 

After these transformations the original sensor vertices can still resolve the new, larger tree, proving that the tree was not the largest possible. We mention that this proof is minimal in the sense that we can only add a vertex when Step C is applicable. When it is not applicable, we are either in the setting of Step A or B, and these transformations are ment to make Step C applicable.

Step D: As a consequence, any optimal tree that we have after Step A must have that the shortest paths between neighboring sensors are all disjoint and contain at most $k+1$ edges, forming the `skeleton' of the tree. Finally, we can calculate the largest number of vertices this skeleton can support, by optimising the distance between two neighboring sensors and the number of vertices on leaf-paths emanating from vertices on these shortest paths between neighboring sensors. 

To obtain the structural lower bound in Theorem \ref{thm:LPLowerBound}, we show that each system of leaf-paths connecting to the same support vertex needs at least as many sensors as in \eqref{eq:LeafPathSensors}. Once these leaf-paths are resolved, the sensors on them can measure some vertices in the rest of the tree, and resolve vertices there, but they cannot measure further than their support vertex would, if it was a sensor. Combining this argument with the lower bound on the number of sensors on the rest of the tree from Theorem \ref{thm:main} allows us to find a lower bound: adding the number of sensors on leaf-paths to the number of sensors the skeleton would need if it was an optimal tree, and then finally subtracting the number of support vertices.

\section{Preliminaries: Attraction of sensors}\label{sec:prelim}

Before the proofs we first introduce some further notions, that not only will be crucial in our proofs, but we believe they could be useful in other contexts as well. As it will become clear later, the construction of the optimal trees is centered around the paths between sensors and the structure on `how' they measure vertices with respect to other sensors, that we call \emph{direct} measuring, and a related notion of  \emph{attraction} of sensors below.

\begin{defn}[Paths]\label{defn:paths}
For a tree $T=(V,E)$ and any pair of distinct vertices $x,y\in V$ we denote the unique path in $T$ between $x$ and $y$ by $\cP_T(x,y)$, its vertex set by $V(\cP_T(x,y))$, and its edge set by $E(\cP_T(x,y))$. We will omit the subscripts when the underlying graph is clear from the context.
\end{defn}

\begin{defn}[Weak and strong sensor paths]\label{defn:sensorpaths}
Given a tree $T=(V,E)$, a set of sensors $S\subseteq V$ and a distinct pair $s_1,s_2\in S$, we call $\cP_T(s_1,s_2)$ a \emph{sensor path} if it does not contain any other sensors beside $s_1$ and $s_2$. $\cP_T(s_1,s_2)$ is called a \emph{strong sensor path} if it is a sensor path and $|E(\cP_T(s_1,s_2))|\le k+1$. A sensor path that is not strong is called \emph{weak}.
\end{defn}

\begin{defn}[Measuring and direct measuring]\label{def:direct-measuring}
Given a tree $T=(V,E)$ and a set of sensors $S\subseteq V$, we say that a sensor $s$ \emph{measures} a vertex $x$ in $T$ if $d_T(s,x)\leq k$. In this case we further say that $s$ \emph{directly measures} $x$, if it also holds that $s'\notin V(\cP_T(s,x))$ for all $s'\in S\setminus\{s\}$.
\end{defn}

We will introduce a concept that we call \emph{minimally resolving} that will be crucial in our proofs. To ease the reader into the fairly non-obvious definition, we start with a simpler definition that is more intuitive: 
\begin{defn}[Resolved within a subset of sensors]\label{defn:resolved-within}
Let $T=(V,E)$ be a tree with a threshold-$k$ resolving set $S$. We say that a vertex $x\in V\setminus S$ is \emph{resolved within} the  sensors $S'=\{s_1, \dots, s_r\} \subseteq S$ in $T$ if \\
(i) $d(s_i, x) \leq k$ for at least one $i = 1,\ldots,r$, \\
(ii) there is \emph{no} sensor in $S\setminus S'$ that directly measures $x$, \\
(iii) for all $y\in V$ for which (ii) holds, there is some $s_i\in S'$ that resolves $x$ and $y$. 
\end{defn}
Note that (ii) is equivalent to the following: every sensor $s^\star$ in $S\setminus S'$ either does not measure $x$ or has $d(s^\star, x)=d(s^\star, s_i)+d(s_i, x)$ for some $i\le r$.
Heuristically, $x$ is resolved within $S'$ if all sensors not in $S'$ can only measure $x$ via a path crossing some sensor in $S'$, and $x$ is distinguished by $S'$ from all other vertices having the same property. Observe that in (iii), the condition that (ii) holds implies that (i) also holds for $y$. Indeed, if (ii) holds for $y$ but (i) does not, then either there is a sensor $s^\star\notin S'$ measuring $y$, such that $d(s', y)\le d(s^\star, y) \le k $ for some $s'\in S'$ by (ii), which is a contradiction, or $y$ cannot be measured by any sensor, and we assumed throughout the paper that we only consider trees where such a vertex is not present in $T$, a contradiction again.

\begin{defn}[`Resolved-within area' of a set of sensors]\label{defn:within}
Let $T=(V,E)$ be a tree with a threshold-$k$ resolving set $S\subseteq V$. The \emph{resolved-within area} of a set of sensors $S'\subseteq S$ is
\[M_T(S')=\{x\in V\setminus S: x \text{ is resolved within } S'\}.\]
\end{defn}
It is not hard to see that the `resolved-within' area is monotone under containment, i.e., when $B\subset S'$, then $M_T(B)\subseteq M_T(S')$, and if $S$ is a threshold-$k$ resolving set for $T$, then $M_T(S)=V\setminus S$.

 The next definition decomposes $M_T(S')$ into \emph{disjoint subsets}: heuristically speaking, starting from the set of single sensors, and increasing the set-size gradually, for a vertex $x$ it finds the minimal set of sensors needed that can distinguish $x$ from all other vertices via direct measuring. 
\begin{defn}[Minimally resolving]\label{defn:unique-measuring}
Let $T=(V,E)$ be a tree with a threshold-$k$ resolving set $S$. A subset of sensors $S' \subseteq S$, $S' =  \{s_1, \ldots, s_r\}$ is said to \emph{minimally resolve} a vertex $x$ in $T$ if \\
(i) $d(s_i, x) \leq k$ for all $i = 1,\ldots,r$, \\
(ii)  there is \emph{no} sensor in $S\setminus S'$ that directly measures $x$, \\
(iii) for all $y\in V$ for which (ii) holds,  there is some $s_i\in S'$ that resolves $x$ and $y$, \\
(iv) all $s_i\in S'$ directly measure $x$, i.e., $d(s_i,x) \neq d(s_j,x) + d(s_i,s_j)$ for $i,j = 1,\dots,r$, $i \neq j$. \\
\end{defn}
For a single sensor $s$, Definitions \ref{defn:resolved-within} and \ref{defn:unique-measuring} are identical. For more sensors, the difference between Definition \ref{defn:within} and \ref{defn:unique-measuring} is in parts (i) and the new criterion (iv):
heuristically, a set of sensors minimally resolve a vertex $x$ if they \emph{all directly measure it}, (i.e, the shortest paths leading to the vertex $x$ do not contain fully each other),  and the set is minimal in the sense that no other sensor directly measures $x$ by part (ii).  Part (iii), similarly to Definition \ref{defn:within}(iii), ensures that an $x\in V$ for which all the conditions hold is distinguished from all vertices $y$ in the resolved-within area of $S'$.

We call the vertices that are minimally resolved by a set of sensors the attraction of the sensor set $S'$, and this is our next definition. 
\begin{defn}[Attraction of a set of sensors]\label{defn:attraction}
Let $T=(V,E)$ be a tree with a threshold-$k$ resolving set $S\subseteq V$. The \emph{attraction} of a set of $r$ sensors $S'=\{s_1,\ldots,s_r\} \subseteq S$ is
\[A_T(s_1,\ldots,s_r)=\{v\in V\setminus S: v \text{ is minimally resolved by } \{s_1, \dots, s_r\}\}.\]
We will omit the subscript $T$ when the underlying graph is clear from the context.
\end{defn}
It is not hard to see that $A_T(s)=M_T(s)$ for a single sensor $s\in S$,  $A_T(S')$ and $A_T(S'')$ are disjoint whenever $S'\neq S''$, and that
\[ M_T(S')=\bigcup_{B\subseteq S' }A_T(B).\]

Heuristically, if a path between a sensor and a vertex does not contain any sensor from $S'$ then non of the vertices of this path can be in the attraction of $S'$. More formally, we will use the following straightforward claim in our proofs:
\begin{claim}\label{claim:measuredpaths}
For a tree $T=(V,E)$ with threshold-$k$ resolving set $S$, let $S'\subseteq S$ be any subset of sensors, and let $s\in S\setminus S'$. Assume that $x\in V$ is measured by $s$, and $V(\cP_T(s,x))$ is disjoint from $S'$. Then $V(\cP_T(s,x))$ is also disjoint from $A_T(S')$.
\end{claim}
\begin{proof}
Assume that $y\in V(\cP_T(x,s))$. Since $x$ is measured by $s$, we have $d_T(s,y)\le d_T(s,x)\le k$. Hence, $y$ is also measured by $s$. Since $s\notin S'$, and $V(\cP_T(s,y))$ does not contain any sensor from $S'$, it follows that $y\notin A_T(S')$, otherwise Definition \ref{defn:unique-measuring}(ii) would be violated.
\end{proof}

Before we continue, we make a few remarks, and a few definitions about the structure of the attraction of two sensors.
\begin{remark}[Size of the attraction of a single sensor]\label{rem:one-attr}
The size of $A_T(s)$, for any sensor $s\in S$, cannot exceed $k$. Indeed, for each distance $1\le j\le k$, there can only be a single vertex at graph distance $j$ from $s$ belonging to $A_T(s)$. Suppose to the contrary that for some $j$ there are two vertices $x, x'\in A_T(s)$ with $d(s,x)=d(s,x')=j$. Then $s$ cannot distinguish between $x,x'$, hence Definition \ref{defn:unique-measuring} (iii) is violated. 
\end{remark}
\begin{remark}[Structure of the attraction of a pair of sensors]\label{rem:two-attr}
If $v$ belongs to  $A_T(s_1,s_2)$ for a pair of sensors $s_1,s_2$, then it is not possible that one of the shortest paths between $x$ to $s_1$ and $x$ to $s_2$ fully contains the other one, otherwise Definition \ref{defn:unique-measuring} (iv) would be violated, so it is not possible that either of the shortest paths from $v$ to $s_1$ and to $s_2$ does not intersect $V(\cP_T(s_1, s_2))\setminus \{s_1,s_2\}$. Hence, there are two possibilities for the location of $v$. Either $v\in V(\cP_T(s_1,s_2))\setminus\{s_1,s_2\}$ or $v$ is connected by a path to a vertex in $V(\cP_T(s_1,s_2))\setminus\{s_1,s_2\}$. Further, there cannot be any third sensor on this path, otherwise $\{s_1,s_2\}$ would not minimally resolve $v$ by Definition \ref{defn:unique-measuring} (ii). 
\end{remark}
Based on this observation, we more generally define a type and a height of a vertex with respect to a pair of sensors $s,s'$, which we will use in Sections \ref{sec:A}--\ref{sec:C}.
\begin{defn}[Type and height with respect to a sensor pair] \label{defn:type-height} Consider a tree $T$ with a threshold resolving set $S$, and let  $s, s'\in S$ and $v\notin \{s, s'\}$. We define 
\begin{equation}\label{eq:typ}
    \mathrm{typ}_{s,s'}(x):=\big(d(x, s)-d(x, s')+d(s,s')\big)/2,
\end{equation}
and we say that $x$ is of type $j$ ($j=\mathrm{typ}_{s,s'}(x)$) with respect to $s,s'$. We further define 
\begin{equation}\label{eq:height}
\mathrm{hgt}_{s,s'}(x):=d_T(x,s)-\mathrm{typ}_{s,s'}(x),
\end{equation}
and we say that it is the height of the vertex $x$ with respect to $s,s'$.
\end{defn}
A short interpretation: $x\in V$ is of type $0$ if the shortest path $\cP_T(x, s')$ from $x$ to $s'$  fully contains $\cP_T(x, s)$, and $x$ is  of type $d(s,s')$ if the situation is reversed. For $1\le j \le d(s,s')-1$, 
 $x\in V$ is of  \emph{type} $j$ if the closest vertex to $x$ on the path $\cP_T(s, s')$ is of distance $j$ from $s$. 
Similarly, the \emph{height} of a vertex $x\in V$ with respect to $s,s'$ is the distance of $x$ from the path $\cP_T(s,s')$.

Based on these definitions, the following claim is a direct consequence.
\begin{claim}[Type-difference]\label{claim:type-difference}
Consider a tree $T$ with a threshold resolving set $S$, and let  $s, s'\in S$ and $x,x'\notin \{s, s'\}$ satisfying that $x$ is measured by $s$ and $x'$ is measured by $s'$.
Then, if additionally $\mathrm{typ}_{s,s'}(x)\neq \mathrm{typ}_{s,s'}(x')$ then the set $\{x,x'\}$ is resolved by $\{s,s'\}$. 
\end{claim}
\begin{proof}
First, if $x$ is measured only by $s$ and not by $s'$, then $d_k(s',x)= k+1 \ge  d_k(s', x')$, hence $s'$ resolves the pair $x,x'$. Analogously, if $x'$ is measured only by $s'$ but not by $s$ then $s$ resolves the pair $x,x'$. The only remaining case is when both vertices are measured by both sensors. In this case, since $x,x'$ have different types, by \eqref{eq:typ}
\[ d_k(x, s) - d_k(x, s')= d(x, s)-d(x, s') \neq d(x',s)- d(x', s') = d_k(x',s)- d_k(x', s'), \]
hence, at least one of $s,s'$ resolves $x,x'$.
\end{proof}
The next three sections are structured as follows. In each of the Sections \ref{sec:A}, \ref{sec:B}, \ref{sec:C}, we introduce a different rewiring procedure called Transformation A, B, C, respectively. Then, we state the main properties of each of these transformations in Sections \ref{sec:A-properties},  \ref{sec:B-properties}, \ref{sec:C-properties}. Then in Sections \ref{sec:A-prelim}, \ref{sec:B-prelim}, \ref{sec:C-prelim} we gather the preliminary claims to be able to prove these properties, and finally in Sections \ref{sec:A-proof}, \ref{sec:B-proof}, \ref{sec:C-proof} we prove that the transformations do exactly what we want. 

\section{Step A: Moving the attractions of single sensors into leaf-paths}\label{sec:A}

To prove Proposition \ref{prop:main}, we start with Step A:  we first show in Lemma \ref{lem:trA} that for every tree $T$ we can rewire a few edges to form another tree $\widehat{T}$ with the same vertex set and threshold-$k$ resolving set, such that every sensor has its attraction contained in a leaf-path starting from itself. This will make further transformations possible that can increase the size of the tree as well.

Recall from Remark \ref{rem:one-attr} that $|A_T(s)|\le k$ for any sensor $s$.
\begin{lemma}\label{lem:trA}
Let $T = (V,E)$ be a tree on $n$ vertices with threshold-$k$ resolving set $S\subseteq V$. For any $s\in S$ let $A_T(s) = \{v_1, v_2,\dots, v_{\ell(s)}\}$ for some $\ell(s)\le k$. Then there exists a tree $\widehat T = (V, \widehat E)$ on the same vertex set in which $S$ is also a threshold-$k$ resolving set, and in which, for each $s\in S$, $A_{\widehat T}(s)\subseteq A_{T}(s)$ with $A_{\widehat T}(s)$ being a \emph{leaf-path} emanating from $s$.
\end{lemma}
\begin{remark}
A consequence of the proof of Lemma \ref{lem:trA} is that if $|A_T(s)|$ is less than $k$ for some $s\in S$, then $T$ cannot be optimal, since we can add extra vertices to the tree to these leaf-paths and the new tree is still resolved. 
\end{remark}
The proof of Lemma \ref{lem:trA} will be based on the following rewiring operation on the tree, which we will call {\it Transformation A}.

\begin{defn}[Transformation A]\label{defn:trA}
Given a tree $T=(V,E)$ with a threshold-$k$ resolving set $S\subseteq V$ and some fixed $s\in S$, let $A_T(s)=\{v_1,\ldots,v_\ell\}$ for some $\ell\le k$. 
Denote the connected components of $T$ spanned on $V\setminus A_T(s)$ by $\widetilde{T_0},\widetilde T_1,\ldots,\widetilde T_r$ for some $r$, where $\widetilde T_0$ contains $s$. For each $1\le i\le r$ let $x_i$ be the (unique) vertex in $\widetilde T_i$ that is closest to $s$ in $T$. Define the edge sets
\begin{align}
E_1&:=\{\{u,v\}\in E(T): u\in A_T(s)\text{ or }v\in A_T(s)\},\label{eq:e1}\\
E_2&:=\Big\{\{s,v_1\}\cup\big(\cup_{i=1}^{\ell-1}\{v_i,v_{i+1}\}\big)\Big\},\label{eq:e2}\\
E_3&:=\{\{s,x_i\}:1\le i\le r\}.\label{eq:e3}
\end{align}
Then define $\tr_A(T,S,s)=(V,E')$ where $E'=(E\setminus E_1)\cup E_2\cup E_3$.
\end{defn}

For an example of Transformation A see Figure \ref{fig:trA}.

\begin{figure}[ht]
    \centering
    \begin{subfigure}[b]{0.49\textwidth}
    \centering
    \includegraphics[width=\textwidth]{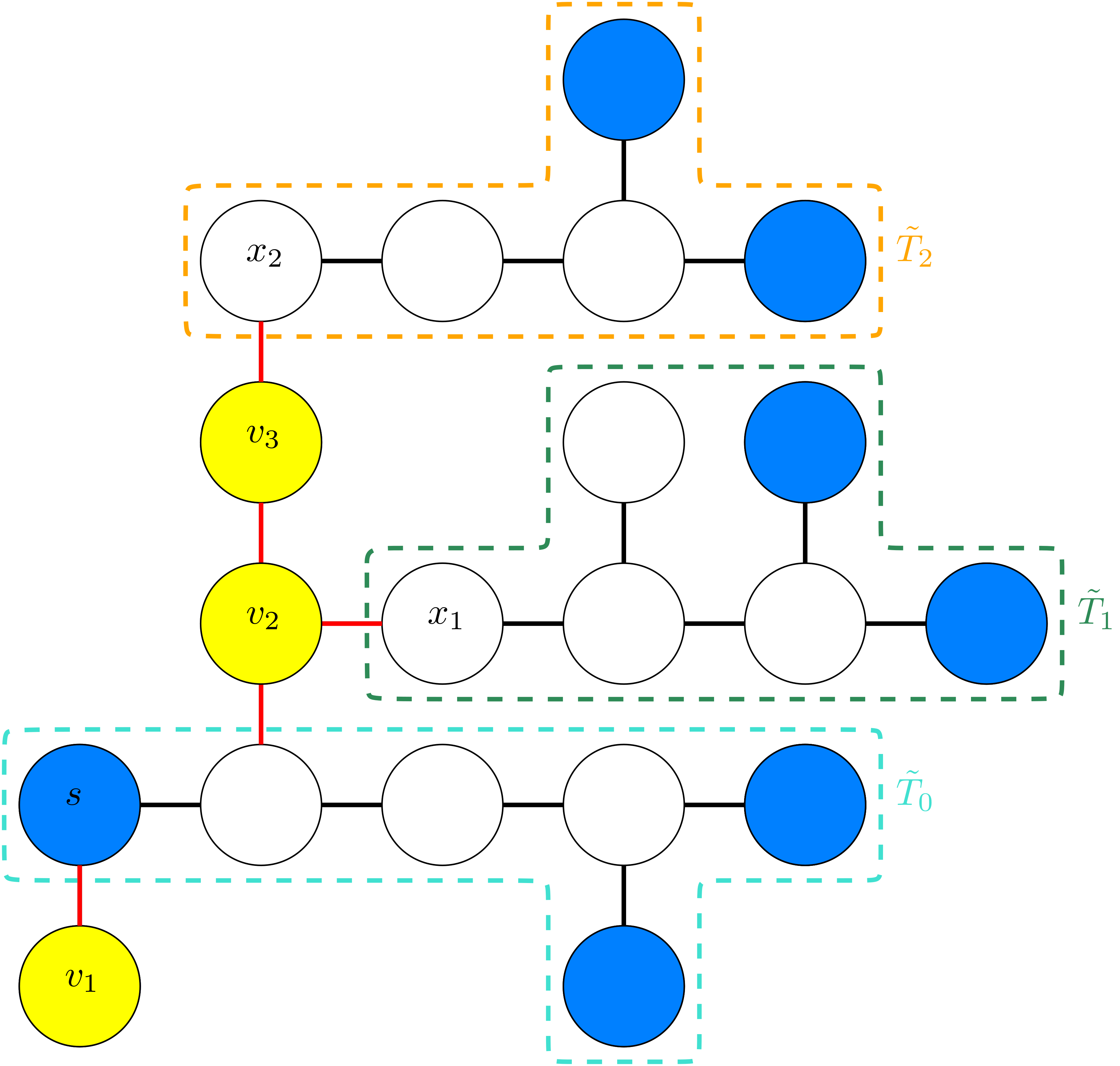}
    \end{subfigure}
    \hfill
    \begin{subfigure}[b]{0.49\textwidth}
    \centering
    \includegraphics[width=\textwidth]{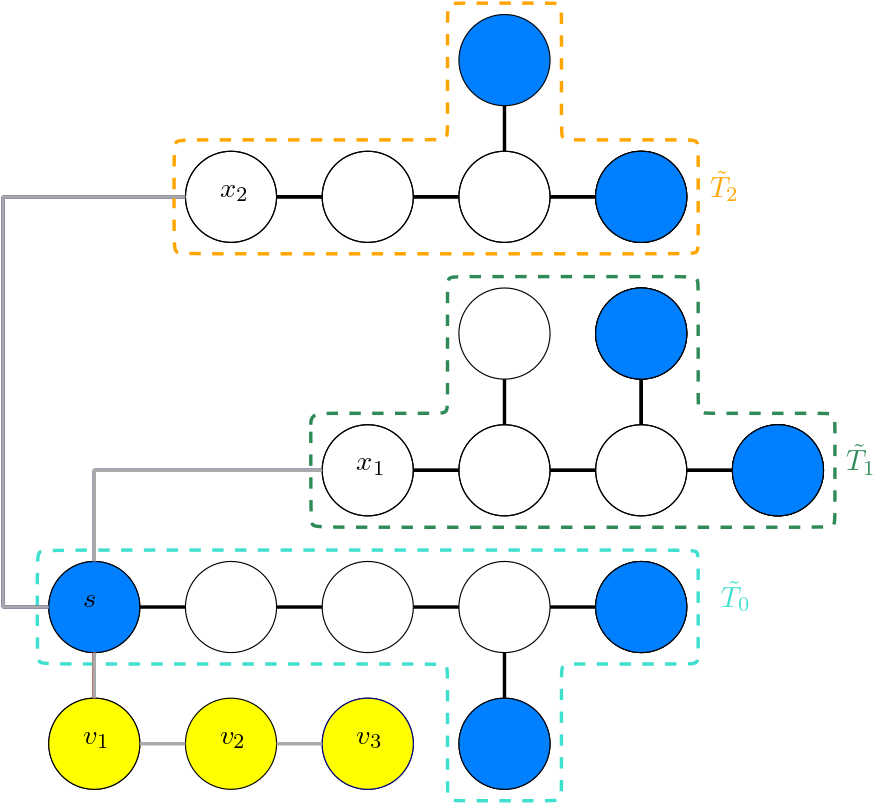}
    \end{subfigure}
    \caption{An example of transformation A with $T$ on the left and $T'=\tr_A(T, S, s)$ on the right. Here $k=3$. The blue vertices are the sensors in $S$, and the yellow vertices are in $A_T(s)$.  The red edges belong to $E_1$, and are deleted by the transformation. The grey edges belong to $E_2\cup E_3$, and are added by the transformation. The subtrees $\wT_0,\wT_1$ and $\wT_2$ are also highlighted.
    As an illustration of Claim \ref{claim:nocomm}, we can observe that sensors in $\wT_1$ and in $\wT_2$ only measure vertices inside their own subtrees, whereas $s$ measures the vertex $x_1$ in $\wT_1$. Furthermore, $s$ is necessary in both $T$ and in $T'$ to distinguish $x_1$ from the other leaf vertex at distance 2 from it in $\wT_1$.}
    \label{fig:trA}
\end{figure}

A couple of comments on this definition: $E_1$ is the set of edges that are adjacent to the vertices in $A_T(s)$ in $T$. After removing $E_1$, $E_2$ rewires $A_T(s)$ into a leaf-path emanating  from $s$ ending at $v_{\ell}$. $E_3$ connects the components on $T$ spanned on $V\setminus A_T(s)$ back together, by connecting $s$ to the (originally) closest vertex $x_i$ in each of the other components $\widetilde T_i$. 
Note that the vertices $x_i$ in the above definition are indeed well-defined, since if some $\widetilde T_i$ had two closest vertices to $s$ in $T$, then they would lie on a cycle in $T$. 

We observe that $\tr_A(T,S,s)$ is indeed a tree, i.e., connected, since the addition of the edge set $E_3$ to $\widetilde{T}$ adds exactly one connection between the components $\widetilde{T}_0$ and $\widetilde{T}_i$ for each $i=1,2,\ldots,r$, and all the vertices in $A_T(s)$ are connected to $s$ via a leaf-path. See Figure \ref{fig:trA} for an illustration.

\subsection{Properties of Transformation A and their consequences}\label{sec:A-properties}
\begin{lemma}[Properties of Transformation A]\label{lem:trAprop}
Let $T=(V,E)$ be a tree with a threshold-$k$ resolving set $S\subseteq V$. Fix some $s\in S$, and consider $T':=\mathrm{tr}_A(T, S, s)$. Then the following hold:
\begin{itemize}
    \item[(i)] $S$ remains a threshold-$k$ resolving set for the tree $T'$.
    \item[(ii)] $A_{T'}(s)=A_T(s)$, and $A_{T'}(s)$ forms a leaf-path emanating from $s$.
    \item[(iii)] For any sensor $s^\star\in S\setminus \{s\}$, if $A_T(s^\star)$ is a leaf-path emanating from $s^\star$ in $T$, then $A_{T'}(s^\star)=A_T(s^\star)$ is still a leaf-path emanating from $s^\star$ in $T'$. 
    \item[(iv)] For any sensor $s^\star\in S\setminus\{ s\}$, if $A_T(s^\star)$ is not a leaf-path emanating from $s^\star$, then $A_{T'}(s^\star)\subseteq A_{T}(s^\star)$.
\end{itemize}

\end{lemma}
Observe that (iii) ensures that attraction of sensors that are already leaf-paths are left untouched by $\tr_A$, while (iv) ensures that for sensors  with attraction that are not (entirely) leaf-paths, $\tr_A$ potentially decreases the number of vertices in the attraction, but never adds new vertices to it. 
\begin{proof}[Proof of Lemma \ref{lem:trA} subject to Lemma \ref{lem:trAprop}]
Let $S=\{s_1,s_2,\ldots,s_m\}$. Let $T_0=T$, and then let us iteratively define $T_i:=\tr_A(T_{i-1},S,s_i)$ for $1\le i\le m$. We prove that $\widehat T=T_m$ satisfies the conditions of the Lemma.
Indeed, for $T_1$ it is true that $A_{T_1}(s_1)=A_T(s_1)$ is a leaf-path emanating from $s_1$ by Lemma \ref{lem:trAprop}(ii), for all other sensors $s_j$, $j\ge 2$, $A_{T_1}(s_j)\subseteq A_T(s_j)$, and $S$ is still a threshold-$k$ resolving set for $T_1$. Inductively we can then assume that in $T_i$, the vertices of $A_T(s_1),\ldots,A_T(s_i)$ already form leaf-paths emanating from $s_1, \ldots, s_i$, respectively, $ A_{T_i}(s_j)\subseteq A_T(s_j)$ for all $j\ge 1$, and $S$ is a threshold-$k$ resolving set in $T_i$. Then $T_{i+1}=\tr_A(T_i, S, s_{i+1})$ moves the vertices $A_{T_i}(s_{i+1})= A_T(s_{i+1})$ into a leaf-path emanating from $s_{i+1}$, and leaves the attraction of sensors $j\le i$ intact, i.e., $A_{T_{i+1
}}(s_j)=A_{T_i}(s_j)$ is a leaf-path emanating from $s_j$ for all $j\le i$ by Lemma \ref{lem:trAprop} (iii). And, for $j\ge i+1$  it holds that $A_{T_{i+1}}(s_j)\subseteq A_{T_i}(s_j) \subseteq A_T(s_j)$ by Lemma \ref{lem:trAprop} (iv) and the inductive assumption. 
This means that in $T_{i+1}$ already $A_T(s_1), \dots, A_T(s_{i+1})$ are all leaf-paths emanating from their respective sensors. Finally $S$ still resolves $T_{i+1}$, hence the induction can be advanced. When $i=m$, the attraction of all sensors have been transformed into leaf-paths, and $S$ still resolves $T_m$, finishing the proof.
\end{proof}
\subsection{Preliminaries to treating Transformation A}\label{sec:A-prelim}
We start with a basic property related to $\tr_A$, ensuring that sensors in the components $\widetilde T_1, \dots \widetilde T_r$ can only measure vertices within their own component.
\begin{claim}[No `communication' between different subtrees]\label{claim:nocomm} Consider the notation of Definition \ref{defn:trA}, and let $T':=\tr_A(T, S, s)$. 
\begin{itemize}
    \item[(i)] Let $s^\star$ be any sensor in $\widetilde T_i$ for some $i\in\{1, \dots, r\}$. Then, for all vertices $y\notin \widetilde T_i$, $d_T(s^\star, y)\ge k+1$ and $d_{T'}(s^\star, y)\ge k+1$ both hold.
    \item[(ii)]Let $s^\star\neq s$ be any sensor in $\wT_0$. Then, for any $y\in \wT_i$, $i\ge 1$, either $d_T(s^\star,y)\ge k+1$, or the shortest path from $s^\star$ to $y$ in $T$ contains $s$. 
\end{itemize}
\end{claim}

\begin{proof}
Proof  of (i): 
Consider a vertex $y\notin \widetilde T_i$. 
Using the notation of Definition \ref{defn:trA}, the vertex closest to $s$ within $\widetilde T_i$ on the path $\cP_{T}(s^\star,s)$ is $x_i$. Since  $x_i\in\wT_i$, the edges of $\cP(s^\star,x_i)$ are present in both $T$ and $T'$. Moreover, $x_i$ has a neighboring vertex $v\in A_T(s)$ in $T$, by the definition of $\tr_A$, see the comments below the definition. $v$ cannot be measured (in $T$) by any sensor in $\wT_i$, otherwise that sensor would measure $v$ via a path not containing $s$, and Definition \ref{defn:unique-measuring}(iii) would be violated as $v\in A_T(s)$. Hence, $d_T(s^\star,v)\ge k+1$, and
\[
d_T(s^\star,x_i)\ge d_T(s^\star,v)-1\ge k+1-1=k.
\]
These two facts will imply both conclusions as we argue next.

First, as $y\notin\wT_i$, the path $\cP_T(s^\star,y)$ contains $v\in A_T(s)$, and we also have that $\cP_T(s^\star,v)$ does not contain $s$. Hence, if $s^\star$ measures $y$ in $T$, then it also measures $v$ in $T$, contradicting Claim \ref{claim:measuredpaths}, as $v\in A_T(s)$. Hence, $d_T(s^\star,y)\ge k+1$ holds.

Next, we will show that $d_{T'}(s^\star,y)\ge k+1$ as well. Since $y\notin\wT_i$ the path $\cP_{T'}(s^\star,y)$ contains $s$, since the only connection between $\widetilde T_i$ and $V\setminus V(\widetilde T_i)$ is the edge $(s, x_i)$ in $T'$.
This implies that 
\[
d_{T'}(s^\star,y)\ge d_{T'}(s^\star,s)=d_{T'}(s^\star,x_i)+1=d_T(s^\star,x_i)+1\ge k+1.
\]

Proof of (ii): Assume that $d_T(s^\star,y)\le k$. Then, since $s^\star\in\wT_0$ and $y\in\wT_i$, $i\ge 1$, the path $\cP_T(s^\star,y)$ contains a vertex  $v\in A_T(s)$. Hence, by Claim \ref{claim:measuredpaths}, $\cP_T(s^\star,y)$ has to contain $s$, otherwise Definition \ref{defn:unique-measuring} (ii) would be violated for $v$ belonging to $A_T(s)$. 
\end{proof}

\subsection{Proof that Transformation A works}\label{sec:A-proof}
\begin{proof}[Proof of Lemma \ref{lem:trAprop}]
{\bf Proof of (i):} Let $x, y\in V$ be a pair of distinct vertices. We shall prove that there is a sensor in $S$ that resolves them in $T'$. We will use the notations of Definition \ref{defn:trA}. We will do a case-distinction analysis with respect to the location of $x$ and $y$ in the components $\wT_i, i\ge 0$ described in the transformation. We start with cases when neither of the vertices are in $A_T(s)$:

{\bf Case 1:} Assume that $x\in\wT_i$ and $y\in\wT_j$ for some $i\ge 1$, $j\ge 0$, $i\ne j$. Then, since $x\notin A_T(s)$, there is a sensor $s'\in S\setminus\{s\}$ that measures $x$ such that $\cP_T(s',x)$ does not include $s$. Then, by Claim \ref{claim:nocomm}(i)--(ii), $s'\in\wT_i$. Therefore, the edges of $\cP_T(s',x)$ are unchanged in $T'$, so $s'$ still measures $x$ in $T'$. However, it does not measure $y\in \wT_j$ in $T'$ by Claim \ref{claim:nocomm}(i). Hence, $s'$ resolves $x$ and $y$ in $T'$.

{\bf Case 2:} Now assume that $x,y\in\wT_i$ for some $i\ge 1$. Let $s'\in S$ be a sensor that resolves $x$ and $y$ in $T$. Then $s'$ has to measure at least one of $x$ and $y$ in $T$, hence $s'\notin\wT_j$ for $j\ge 1$, $j\ne i$ by Claim \ref{claim:nocomm}(i). There are two (sub)cases: either $s'\in\wT_i$, or $s'\in \wT_0$. First we consider $s'\in \wT_i$.
Then the paths $\cP_T(s',x)$, $\cP_T(s',y)$ do not contain any vertex in $A_T(s)$ (since both endpoints of both paths belong to $\wT_i$). Therefore, the edges of $\cP_T(s',x)$ and $\cP_T(s',y)$ are all still present in $T'$, and $s'$ resolves $x$ and $y$ in $T'$.

If $s'\in \wT_0$, then either $s'=s$ or $s'\neq s$. We start with the case $s'=s$. Recall $x_i$ from Definition \ref{defn:trA}.
Then $\{s,x_i\}\in E_3$ (an edge added to create $T'$). Since every path $\cP_T(s,v)$, $v\in\wT_i$ starts with the segment $\cP_T(s,x_i)$ in $T$, that we replaced with the single edge $\{s, x_i\}$ to obtain $\cP_{T'}(s,v)$,  the following holds for all $v\in \wT_i$:
\[d_{T'}(s,v)=d_T(s,v)-d_T(s,x_i)+1,\]
 Hence, the difference of the distances does not change:
\[
d_{T'}(s,x)-d_{T'}(s,y)=d_T(s,x)-d_T(s,y)\neq 0,
\]
the latter being nonzero by the assumption that $s$ resolves $x,y$ in $T$. Since these distances in $T'$ are less than in $T$, $s$ still resolves $x,y$ in $T'$.

The last possibility is that $s'\in\wT_0\setminus\{s\}$. Then $s'$ has to measure at least one of $x$ and $y$ in $T$, say it measures $x$. Then, by Claim \ref{claim:nocomm}(ii), $\cP_T(s',x)$ contains $s$. On the path $\cP_T(s,x)$, there has to be at least one vertex in $A_T(s)$ (since $x\in\wT_i$), let the closest one to $x$ be $u$. Then, since $\wT_i$ is a connected component in $V\setminus A_T(s)$, and $y\in\wT_i$ (and $T$ is a tree), $u$ is also on the path $\cP_T(s',y)$. Hence, $\cP_T(s',s)\subseteq\cP_T(s',u)\subseteq\cP_T(s',x)\cap\cP_T(s',y)$. This implies that
\[
    d_T(s,x)=d_T(s',x)-d_T(s',s),\quad \text{and} \quad
    d_T(s,y)=d_T(s',y)-d_T(s',s).
\]
Therefore, if $s'$ resolves $x$ and $y$ in $T$, then $s$ also resolves them in $T$, and then the reasoning of the previous paragraph applies, and so $s$ resolves $x,y$ also in $T'$.

{\bf Case 3:} Next, suppose that $x,y\in\wT_0$, and let $s'$ be a sensor that resolves them in $T$. Then $s'$ measures at least one of $x$ and $y$ in $T$, which, by Claim \ref{claim:nocomm}(i) is only possible if $s'\in\wT_0$. (Here $s'$ may or may not be $s$.) This implies that the edges of the paths $\cP_T(s',x)$ and $\cP_T(s',y)$ stay intact in $T'$, thus they are still distinguished by $s'$ in $T'$.

We continue with cases for which at least one of $x,y\in A_T(s)$:

{\bf Case 4:} If $x$ and $y$ are both in $A_T(s)$, then in $T'$ they will both be part of a single leaf-path (of length at most $k$) emanating from $s$, hence $s$ will resolve them in $T'$.

{\bf Case 5:} The next case is when $x\in A_T(s)$ and $y\in\wT_i$ for some $i\ge 1$. Since $y\notin A_T(s)$, there has to be a sensor $s'$ that measures $y$ such that $\cP_T(s',y)$ does not contain $s$. Further, since $y\in\wT_i$ and by Claim \ref{claim:nocomm}(ii), this is only possible if $s'\in\wT_i$. Then the vertices of the path $\cP_T(s',y)$ remain unchanged in $T'$, so $s'$ still measures $y$ in $T'$. However, it cannot measure $x\in A_T(s)$ in $T'$ by Claim \ref{claim:nocomm}(i). Thus, $s'$ distinguishes $x$ and $y$ in $T'$.

{\bf Case 6:} The final case is when $x\in A_T(s)$ and $y\in\wT_0$. Then there exists a sensor $s'\in S\setminus\{s\}$ which measures $y$ such that the path $\cP_T(s',y)$ does not contain $s$. By Claim \ref{claim:nocomm} (ii), $s'\in\wT_0$. Then there are two possible subcases.

First assume that $V(\cP_T(s,s'))\cap V(\cP_T(s',y))=\{s'\}$, i.e., $s,s',y$ all lie on a path in $\wT_0$ in this order. Then in $T'$, $x$ is added to a leaf-path emanating from $s$, and the edges of $\cP_T(s,s')$, $\cP(s',y)$ are unchanged, hence $x,s,s',y$ will all lie on a single path in $T'$ in this order. Then, since $x$ is measured by $s$ and $y$ is measured by $s'$, and $\mathrm{typ}_{s,s'}(x)=0$ while $\mathrm{typ}_{s,s'}(y)=d_{T'}(s,s')$, by Claim \ref{claim:type-difference}, at least one of $s,s'$ resolves $x,y$.

Next, assume that $V(\cP_T(s,s'))\cap V(\cP_T(s',y))\ne\{s'\}$. Then, in fact, $V(\cP_T(s,s'))\cap V(\cP_T(s',y))=V(\cP_T(s',v))$ for some $v\ne s'$, i.e., $\mathrm{typ}_{s,s'}(y)=\mathrm{typ}_{s,s'}(v)\neq 0$. For this we also know that $v\ne s$, since $\cP_T(s',y)$ does not contain $s$). Observe that 
 the edges of the paths $\cP_T(s,v)$, $\cP_T(s',v)$, $\cP_T(y,v)$ are unchanged in $T'$. We also know that $\mathrm{typ}_{s,s'}(x)=0$ in $T'$ since $x$ is in a leaf-path emanating from $s$ in $T'$. Hence, Claim \ref{claim:type-difference} applies with $y:=x'$ so at least one of $s, s'$ resolves $x,y$.

{\bf Proof of (ii):} Using the notation of Definition \ref{defn:trA}, let $A_T(s)=\{v_1,\ldots,v_\ell\}$. We claim that $A_{T'}(s)=\{v_1, \dots, v_\ell\}$. Indeed, $\{v_1, \dots, v_\ell\}\subseteq A_{T'}(s)$ since all these vertices are measured by $s$ in $T'$, and $(s, v_1,\ldots,v_\ell)$ forms a leaf-path in $T'$, so the shortest path from every other sensor $s^\star\in S$ to any $v_i$ contains the sensor $s$ and hence Definition \ref{defn:unique-measuring} (iii) would be violated otherwise.

In order to show that $A_{T'}(s)\subseteq\{v_1, \dots, v_\ell\}$, consider a vertex $x\in V\setminus (A_T(s)\cup S)$. Since $x\notin A_T(s)$, there exists (at least one) $s^\star\in S\setminus\{s\}$ that directly measures $x$. We will show that in this case 
\begin{equation}\label{eq:other-sensor-measuring-x}
    d_{T'}(s^\star,x)\le k \quad \text{ and }\quad s\notin V(\cP_{T'}(s^\star,x)),
\end{equation} hence $x$ cannot be in $A_{T'}(s)$, either. Consider the path $\cP_T(s^\star,x)$. Since it has length at most $k$, and it does not contain $s$, none of its vertices can be in $A_T(s)$ by Claim \ref{claim:measuredpaths}. It follows that $\tr_A(T, S, s)$ does not remove any of the edges in $\cP_T(s^\star,x)$, since none of them are adjacent to any vertex in $A_T(s)$.  Hence, all the edges of $\cP_T(s^\star,x)$ are present in $T'$. As a result, \eqref{eq:other-sensor-measuring-x} holds and $x\notin A_{T'}(s)$.

Finally, $A_{T'}(s)=A_T(s)$ forms a leaf-path emanating from $s$ in $T'$, because $E_2$ forms exactly that leaf-path in \eqref{eq:e2}, finishing the proof of (ii). 

{\bf Proof of (iii):} Since the vertices of $A_T(s^\star)$ are all on a leaf-path emanating from $s^\star$, the shortest path between any of them and $s$ contains $s^\star$, hence,  these vertices, (including $s^\star$), cannot belong to $A_T(s)$. Therefore, when executing Transformation A at $s$, all the edges between vertices of $A_T(s^\star)\cup\{s^\star\}$  remain intact. Moreover, $d_T(s,s^\star)\le d_T(s,x)$ for any $x\in A_T(s^\star)$, so such $x$ is not the closest vertex in the component to $s$, hence in the edge set $E_3$ of \eqref{eq:e3} none of the edges connect to any $x\in A_T(s^\star)$. This shows that $A_T(s^\star)$ still forms a leaf-path in $T'$ emanating from $s^\star$, and $A_{T}(s^\star)\subseteq A_{T'}(s^\star)$. The fact that $A_{T'}(s^\star)\subseteq A_{T}(s^\star)$ is proved in the same way as part (iv) below.

{\bf Proof of (iv):} For an indirect proof, let us assume that there is a vertex $x\in A_{T'}(s^\star)$, such that $x\notin A_T(s^\star)$, i.e., a new vertex is added to the attraction of sensor $s^\star$ because of the transformation. We observe that $x\in A_{T'}(s^\star)$ implies by Definition \ref{defn:unique-measuring}(i) that the path $\cP_{T'}(s^\star,x)$ has length at most $k$, and it does not contain $s$, or any other sensor besides $s^\star$, by Definition \ref{defn:unique-measuring}(ii) and (iii). This means that the edges of the path $\cP_{T'}(s^\star,x)$ can be neither in $E_3$ nor in $E_2$. That is, all edges of $\cP_{T'}(s^\star,x)$ are also present in $T$. If despite this $x\notin A_T(s^\star)$, it is only possible if there is a sensor $s'\in S\setminus\{s^\star\}$ such that the path $\cP_T(s',x)$ has length at most $k$, and it contains no sensors besides $s'$ in $T$. Then there are the following two possibilities.

If $s'\ne s$, then none of the vertices in $\cP_T(s',x)$ can belong to $A_T(s)$ because the path $\cP_T(s',x)$ does not contain $s$ by assumption. Hence none of the edges of $\cP_T(s',x)$ is in $E_1$, i.e., the transformation leaves these edges untouched. Therefore, $\cP_T(s',x)$ is still a path in $T'$ (of length at most $k$), contradicting the assumption that $x\in A_{T'}(s^\star)$, since both $s^\star$ and $s'$ directly measure $x$ in $T'$.

If $s'=s$, and none of the vertices of $\cP_T(s,x)$ is in $A_T(s)$, then $E(\cP_T(s,x))$ is disjoint from $E_1$, hence $\cP_T(s,x)$ is still a path in $T'$ (of length at most $k$). Since $s^\star\notin V(\cP_T(s,x))$, this contradicts the assumption that $x\in A_{T'}(s^\star)$.

The only remaining case is that $s'=s$, and at least one vertex in $V(\cP_T(s,x))$ belongs to $A_T(s)$. This means that $s$ and $x$ get disconnected in $\widetilde{T}=(V,E\setminus E_1)$, i.e., we may assume $x\in \widetilde T_i$ for some $i\ge 1$.   In this case, the edge set $E_3$ contains an edge from $s$ to $x_i$ (see Def. \ref{defn:trA}. This vertex $x_i$ must lie on the path $\cP_T(x,s)$, otherwise there would be a cycle in $T$. Hence, $V(\cP_{T'}(s,x))\subseteq V(\cP_T(s,x))$. Therefore, the path $\cP_{T'}(s,x)$ is of length at most $k$, and it does not contain any other sensor besides $s$, contradicting the assumption that $x\in A_{T'}(s^\star)$.
This shows that $A_{T'}(s^\star)\subseteq A_T(s^\star)$ for all $S\setminus \{s\}$, and  finishes the proof of (iv).
\end{proof}

Lemma \ref{lem:trA} shows us that it is sufficient to consider optimal trees that have the attractions of single sensors contained in leaf-paths attached to the corresponding sensors. Next we analyze the arrangement of the attraction of \emph{pairs} of sensors in an optimal tree.

\section{Step B: Shortening too long sensor paths}\label{sec:B}

Recall from Definition \ref{defn:sensorpaths} that sensor paths are strong (respectively, weak) if their length is at most $k+1$ (respectively, at least $k+2$). Our next step will be to show that weak sensor paths can be shortened as long as all the strong sensor paths are disjoint.
First we state the conditions under which the next transformation applies. The idea is that after the repetitive execution of transformation A in Lemma \ref{lem:trA}, either these conditions are satisfied or we may directly jump to Transformation C in Section \ref{sec:C} below.  
\begin{cond}\label{cond:trB}
Let $T=(V,E)$ be a tree on $n$ vertices with a threshold-$k$ resolving set $S$. Suppose that the following hold:
\begin{itemize}
\item[(i)]
for all sensors $s\in S$ the attraction $A_T(s)$ is contained in a single leaf-path starting from $s$,
\item[(ii)]
any pair of strong sensor paths are disjoint, possibly except for their endpoints, and
\item[(iii)]
there is at least one weak sensor path in $T$, and  $\cP_T(s_0,s_1)$ is one of the longest ones.
\end{itemize}
\end{cond}

\begin{remark}\label{rem:fourpaths}\normalfont
Condition (ii) above is equivalent to the following: there are no two strong sensor paths in $T$ that share an edge. This can be seen as follows. Assume $\cP_T(s_1,s_2)$ and $\cP_T(s_3,s_4)$ are two distinct strong sensor paths that do not share an edge, but do share a vertex $v$ that is not the endpoint of either of them. If the $s_i$-s are not all different, say $s_4=s_2$, then let $w$ be the vertex neighboring $s_2$ on the path $\cP_T(s_1,s_2)$. Then $\cP_T(s_1,s_2)$ and $\cP_T(s_2,s_3)$ share the edge $\{w,s_2\}$, a contradiction. Hence, $s_1,s_2,s_3,s_4$ are indeed all different. Let $s'_1, s'_2, s'_3, s'_4$ be a relabeling of $s_1,s_2,s_3,s_4$ such that $s'_i$ ($i=1,2,3,4$), is the $i$-th closest sensor to $v$ among $s_1,s_2,s_3,s_4$, breaking ties arbitrarily. Let $u$ be the vertex neighboring $v$ on $\cP_T(s'_1,v)$. Then $\cP_T(s'_1,s'_2)$ and $\cP_T(s'_1,s'_3)$ share the edge $\{u,v\}$, and both are strong sensor paths, since
\begin{align*}
\max\{d_T(s'_1,s'_2),d_T(s'_1,s'_3)\}&=
\max\{d_T(s'_1,v)+d_T(v,s'_2),d_T(s'_1,v)+d_T(v,s'_3)\}\\
&\le \max\{d_T(s_1,v)+d_T(v,s_2),d_T(s_3,v)+d_T(v,s_4)\}\\
&=\max\{d_T(s_1,s_2),d_T(s_3,s_4)\}\le k+1.
\end{align*}
Hence, these two properties are indeed equivalent, and we will use both formulations interchangeably later in this paper.
\end{remark}
The main lemma of this section is the following:
\begin{lemma}\label{lem:trB}
Let $T=(V,E)$ be a tree with a threshold-$k$ resolving set $S$. Suppose that Condition \ref{cond:trB}(i)--(iii) hold for $T$. 
Then there is another tree $\widehat T=(V,\widehat E)$ on the same vertex set such that the following hold:
\begin{itemize}
\item[(i)] $S$ is still a threshold-$k$ resolving set in $\widehat T$,
\item[(ii)] for each $s\in S$, $A_{\widehat T}(s)=A_T(s)$, and its vertices still form a leaf-path in $\widehat T$ emanating from $s$, and
\item[(iii)] there is at least one pair of strong sensor paths in $\widehat T$ that share an edge.
\end{itemize}

\end{lemma}

The proof of Lemma \ref{lem:trB}  relies on another operation on the graph, given below, which we call {\it Transformation B}.
Before the definition, recall Definition \ref{defn:attraction} and Remark \ref{rem:two-attr} about the structure of the attraction of two sensors, as well as Definition \ref{defn:type-height} about the types and heights of a vertex $x$ with respect to two sensors $s,s'$. With regard to this, we make some comments next.

Note that $(d_T(x,s),d_T(x,s'))$ is a one-to-one function of $(\mathrm{typ}_{s,s'}(x),\mathrm{hgt}_{s,s'}(x))$, implying that there cannot be two vertices $x,y\in A_T(s,s')$ with $\mathrm{typ}_{s,s'}(x)=\mathrm{typ}_{s,s'}(y)$ and $\mathrm{hgt}_{s,s'}(x)=\mathrm{hgt}_{s,s'}(y)$, since then no sensor would resolve them.
For any $1\le j\le d(s,s')-1$ there can be at most $k-\max\{j,d(s,s')-j\}+1$ vertices in $A_T(s,s')$ with type $j$ as the possible pairs of distances of a type $j$ vertex from $s$ and $s'$, respectively, are $(j,d(s,s')-j),(j+1,d(s,s')-j+1),\ldots$ with the pair of largest distances being either $(k,d(s,s')+k-2j)$ or $(2j+k-d(s,s'),k)$, depending on whether $j\ge (d(s,s'))/2$ or not (these pairs of distances correspond to heights $0,1,\ldots,k-\max\{j,d(s,s')-j\}$). Observe that it could happen that a type-$j$ vertex is not in $A_T(s, s')$ for some $j \in \{1, \dots, d(s, s')-1\}$ even if its height is at most $k-\max\{j,d(s,s')-j\}$, namely, when that vertex is directly measured by a third sensor $s''$.

\begin{defn}[Transformation B]\label{defn:trB}
Let $T=(V,E)$ be a tree with a threshold-$k$ resolving set $S\subseteq V$ and $s_0, s_1\in S$ such that Conditions \ref{cond:trB}(i)--(iii) hold. Let $w_1\in V(\cP_T(s_0,s_1))$ be the vertex for which $\{s_0,w_1\}\in E$. Since $w_1\not\in A_T(s)$ and Condition \ref{cond:trB}(ii) holds, there is a unique other sensor $s'_0\neq s_1$ that directly measures $w_1$. Let $w_1,w_2,\ldots,w_{d_T(s_0,s'_0)-1}$ be the vertices in $V(\cP_T(s_0,s'_0))\setminus\{s_0,s'_0\}$ in order of increasing distance from $s_0$, and let
\[
V(\cP_T(s_0,s_1))\cap V(\cP_T(s_0,s'_0))=\{w_1,w_2,\ldots,w_q\}.
\]
Furthermore, let $u_{q+1}$ be the vertex of $\cP_T(w_q,s_1)$ for which $\{w_q,u_{q+1}\}\in E(\cP_T(w_q,s_1))$.
Now define the vertex sets
\begin{align*}
V_1&:=\{v\in A_T(s_0,s'_0):\ \typ_{s_0,s'_0}(v)=q,\  \hgt_{s_0,s'_0}(v)\ge 1,\ u_{q+1}\in V(\cP_T(w_q,v))\},\\
V_2&:=\{v\in V:\ \typ_{s_0,s'_0}(v)=q,\ \hgt_{s_0,s'_0}(v)\ge 1,\ u_{q+1}\notin V(\cP_T(w_q,v))\}
\end{align*}
with $V_1=:\{v^{(1)},v^{(2)},\ldots,v^{(|V_1|)}\}$ and $V_2=:\{v^{(|V_1|+1)},v^{(|V_1|+2)},\ldots,v^{(|V_1|+|V_2|)}\}$, and the edge sets
\begin{align*}
E_1&:=\{\{u,v\}\in E: u\in V_1\cup V_2\text{ or }v\in V_1\cup V_2\}\cup\{w_q,u_{q+1}\},\\
E_2&:=\{\{w_q,v^{(1)}\}\}\bigcup\big(\cup_{i=1}^{|V_1|+|V_2|-1}\{\{v^{(i)},v^{(i+1)}\}\}\big).
\end{align*}
Consider the subgraph $\wT=(V\setminus (V_1\cup V_2),E\setminus E_1)$, and denote its connected components by $\widetilde T_0,\widetilde T_1,\ldots,\widetilde T_r$ where $\widetilde T_0$ contains $s_0$ (and $s'_0$), and $\wT_1$ contains $s_1$. For each $1\le i\le r$ let $x_i$ be the vertex in $\widetilde T_i$ that is closest to $s_0$ in $T$. 
Define the third edge set
\begin{align*}
E_3:=&\{\{s_0,x_i\}:1\le i\le r\}.
\end{align*}
Finally, define $\tr_B(T,S,s_0,s_1)=(V,E')$ where $E'=(E\setminus E_1)\cup E_2\cup E_3$.
\end{defn}

For an example of Transformation B see Figure \ref{fig:trB}.

\begin{figure}[ht]
    \centering
    \begin{subfigure}[b]{0.49\textwidth}
    \centering
    \includegraphics[width=\textwidth]{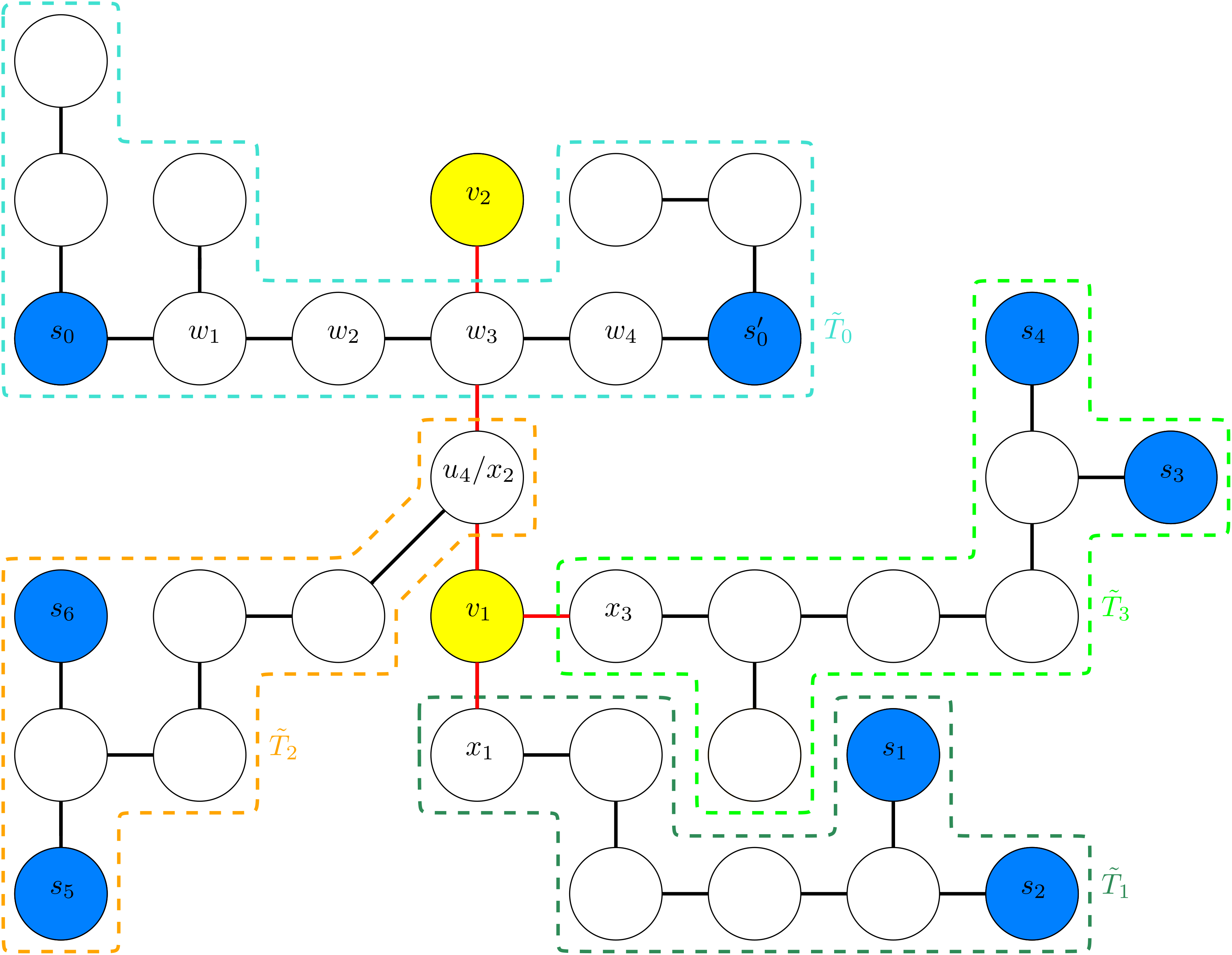}
    \end{subfigure}
    \hfill
    \begin{subfigure}[b]{0.49\textwidth}
    \centering
    \includegraphics[width=\textwidth]{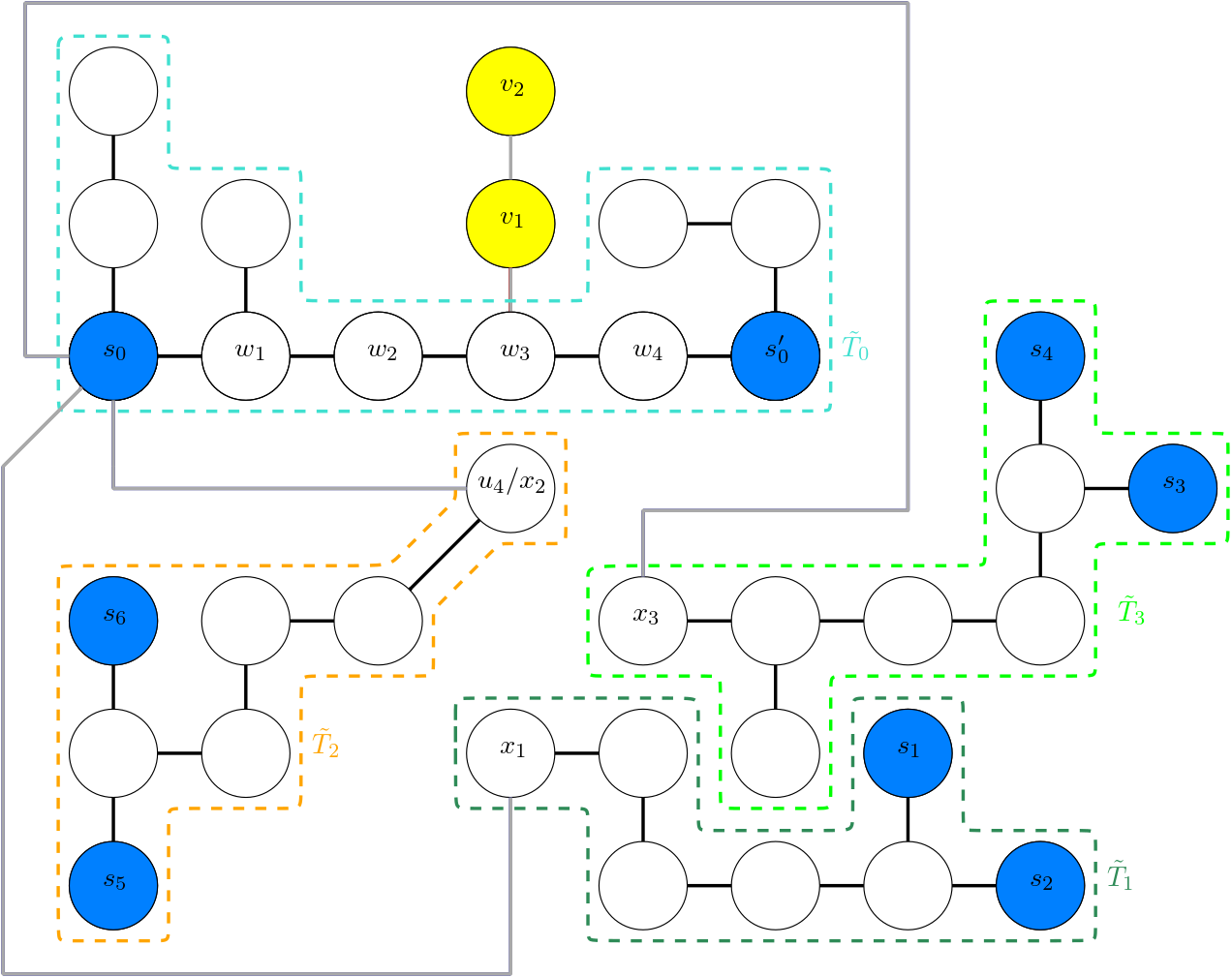}
    \end{subfigure}
    \caption{An example of transformation B with $T$ on the left and $T'=\tr_B(T, S, s_0, s'_0)$ on the right. Here $k=5$, and $q=3$. The blue vertices are the sensors in $S$, and the yellow vertices are in $V_1\cup V_2$: $v_1\in V_1, v_2\in V_2$. The red edges belong to $E_1$, and are deleted by the transformation. The grey edges belong to $E_2\cup E_3$, and are added by the transformation. The subtrees $\wT_0,\wT_1,\wT_2$ and $\wT_3$ are also highlighted. Note that $x_3$ and the leaf vertex at distance two from $x_3$ within $\widetilde T_3$ are only resolved by $s_0'$ in $T$, and by $s_0$ in $T'$. While the sensors $s_0, s_0'$ may measure vertices in the subtrees $\widetilde T_1, \widetilde T_2, \widetilde T_3$, any sensor in $\widetilde T_i, i\ge 1$ can measure vertices only within its own subtree and possibly in $\widetilde T_0$ (this latter case is not depicted in the picture, but could happen e.g. if the sensors $s_5, s_6$ were closer to $u_4$).}
    \label{fig:trB}
\end{figure}

A couple of comments on this definition:
$s_0'$ cannot be equal to $s_1$, since $d(s_0, s_1)\ge k+2$ by the assumption that $\cP_T(s_0, s_1)$ is a weak sensor path. The sensor $s_0'$ is indeed unique, since if there was another sensor $s_0''$ also directly measuring $w_1$ in $T$, then the two sensor paths $\cP_T(s_0, s_0'), \cP_T(s_0, s_0'')$ would both be strong and violate Condition \ref{cond:trB}(ii). This uniqueness of the sensor $s_0'$ means that $w_1\in A_T(s_0,s_0')$. A similar argument shows that $w_{d_T(s_0,s'_0)-1}$ is also in $A_T(s_0,s_0')$.

$E_1$ is the set of edges that are adjacent to the vertices in $V_1\cup V_2$ in $T$, plus the edge $\{w_q,u_{q+1}\}$ (in case $u_{q+1}\notin V_1$, that is, $u_{q+1}\notin A_T(s_0,s'_0)$). The point of removing $E_1$ from the graph is to rewire the edges (with the addition of $E_2$ and $E_3$) such that the path between $s_0$ and $s_1$ becomes shorter while ensuring that the vertices of $A_T(s_0,s'_0)$ are still identified by the sensors. The removal of $E_1$, and then the addition of $E_2$ rewires $A_T(s_0, s_0')$ as follows: those vertices in $A_T(s_0,s'_0)$ that are not of type $q$ stay 'at their place' (relative to $s_0$ and $s'_0$). Those that are type $q$, are rewired into a single leaf-path emanating from $w_q$. We will show that $V_2$ must be already a leaf-path in $T$, and thus we essentially append $V_1$ to the end of this path.
We will show in Claim \ref{claim:typq} that this leaf-path contains at most $k-q$ vertices besides $w_q$, which ensures that all of them will be measured by both $s_0$ and $s'_0$ after the transformation. After all this, $E_3$ connects the components of $(V,(E\setminus E_1)\cup E_2)$ back together, by connecting $s_0$ to the (originally) closest vertex $x_i$ in each of the other components $\widetilde T_i$. 
Note that the vertices $x_i$ in the above definition are indeed well-defined, since if some $\widetilde T_i$ had two closest vertices to $s_0$ in $T$, then they would lie on a cycle in $T$, contradicting the tree property, similarly as before for Transformation A.

We observe that $\tr_B(T,S,s_0,s_1)$ is indeed a tree, i.e., connected, since the addition of the edge set $E_3$ to $\widetilde{T}$ adds exactly one connection between the components $\widetilde{T}_0$ and $\widetilde{T}_i$ for each $i=1,2,\ldots,r$, and the addition of $E_2$ adds the vertices of $V_1\cup V_2$ to $\wT_0$ as a single leaf-path. See Figure \ref{fig:trB} for an illustration.

\subsection{Properties of Transformation B and their consequences}\label{sec:B-properties}
\begin{lemma}[Properties of Transformation B]\label{lem:trBprop}
Let $T=(V,E)$ be a tree with a threshold-$k$ resolving set $S\subseteq V$ and $s_0, s_1\in S$, for which Condition \ref{cond:trB}(i)--(iii) hold,  and consider $T':=\tr_B(T, S, s_0, s_1)$. Then the following hold:
\begin{itemize}
\item[(i)]
$S$ remains a threshold-$k$ resolving set for $T'$,
\item[(ii)]
for each $s\in S$, $A_{T'}(s)=A_T(s)$, and its vertices still form a leaf-path in $T'$ emanating from $s$,
\item[(iii)]
for each pair of sensors $s,s'\in S$, if $\cP_{T'}(s,s')$ is a sensor path, then $\cP_T(s,s')$ was also a sensor path (in $T$), and $d_{T'}(s,s')\le d_T(s,s')$,
\item[(iv)]
$\cP_{T'}(s_0,s_1)$ is still a sensor path, and is strictly shorter than $\cP_{T}(s_0,s_1)$,
\item[(v)]
if $\cP_{T'}(s_0,s_1)$ is a strong sensor path, then $T'$ has a pair of strong sensor paths that share an edge.
\end{itemize} 
\end{lemma}

\begin{proof}[Proof of Lemma \ref{lem:trB} subject to Lemma \ref{lem:trBprop}]
Let $T_0=T$, and then let us iteratively define $T_i:=\tr_B\big(T_{i-1},S,s_0^{(i-1)},s_1^{(i-1)}\big)$ for $i\ge 1$, as long as $T_{i-1}$ does not have a pair of strong sensor paths that share an edge, and where $s_0^{(i-1)}$, $s_1^{(i-1)}$ are the endpoints of one of the longest weak sensor paths in $T_{i-1}$. Let $\widehat T=T_{i_{\max}}$ where $i_{\max}$ is the first index $i$ in this procedure for which there is a pair strong sensor paths in $T_i$ that share an edge. We prove that this procedure is well-defined, and $\widehat T=T_{i_{\max}}$ satisfies the conditions of the Lemma.

For an inductive proof, assume that $S$ is a threshold-resolving set in $T_{i-1}$, $1\le i\le i_{\max}$, and Condition \ref{cond:trB}(i)--(iii) all hold for $T_{i-1}$. (These indeed hold for $i=1$.) Then $T_i:=\tr_B\big(T_{i-1},S,s_0^{(i-1)},s_1^{(i-1)}\big)$ is well-defined.

By Lemma \ref{lem:trBprop}(i), if $S$ was a threshold-$k$ resolving set in $T_{i-1}$, then it will remain so in $T_i$.
By Lemma \ref{lem:trBprop}(ii), if Condition \ref{cond:trB}(i) held for $T_{i-1}$, then it will also hold in $T_i$. Furthermore, for every sensor $s\in S$, $A_{T_i}(s)=A_{T_{i-1}}(s)$, and its vertices still form a leaf-path emanating from $s$ in $T_i$.
Condition \ref{cond:trB}(ii) holds for $T_i$ by assumption when $i\le i_{\max}-1$.
By Lemma \ref{lem:trBprop}(iv)--(v) either $\cP_{T_i}(s_0^{(i-1)},s_1^{(i-1)})$ is still a weak sensor path in $T_i$, and then Condition \ref{cond:trB}(iii) holds for $T_i$, or $\cP_{T_i}(s_0^{(i-1)},s_1^{(i-1)})$ is a strong sensor path in $T_i$, and then $T_i$ has a pair of strong sensor paths sharing an edge, meaning $i=i_{\max}$.
This finishes the proof that $T_{i}:=\tr_B\big(T_i,S,s_0^{(i-1)},s_1^{(i-1)}\big)$ is indeed well-defined for $i\le i_{\max}$, and inductively shows that parts (i), (ii) of Lemma \ref{lem:trB} hold for $T_i$ for any $i\le i_{\max}$. We are only left to show that the procedure finishes in finitely many steps, that is, $i_{\max}<\infty$. In that case, by assumption part (iii) of Lemma \ref{lem:trB} also holds for $\widehat T=T_{i_{\max}}$.

By Lemma \ref{lem:trBprop}(iii), for $i\le i_{\max}$, each sensor path in $T_{i-1}$ either stops being a sensor path in $T_i$, or otherwise its length does not increase. Also, new sensor paths cannot emerge in $T_i$ compared to $T_{i-1}$. Moreover, by part (iv) of the same lemma, the length of at least one sensor path strictly decreases from $T_{i-1}$ to $T_i$. Therefore, if $\Sigma_i$ is the sum of the lengths of all sensors paths in $T_i$, then $\Sigma_{i}\le\Sigma_{i-1}-1$ for every $i\le i_{\max}$. But $\Sigma_i\ge 0$ has to hold for every $i$, which implies that $i_{\max}<\infty$, finishing the proof.
\end{proof}

\subsection{Preliminaries to treating Transformation B}\label{sec:B-prelim}
Similarly to transformation A, we start by proving a 'no communication' lemma for the graph components in $\tr_B$.
\begin{claim}[No `communication' between different subtrees]\label{claim:nocomm-B} Consider the notation of Definition \ref{defn:trB}, and let $T':=\tr_B(T, S, s_0, s_1)$. 
\begin{itemize}
    \item[(i)] Let $s^\star$ be any sensor in $\widetilde T_i$ for some $i\ge 1$. Then, for all vertices $y\in \widetilde T_j$, $j\ge 1$, $j\ne i$, we have that $d_T(s^\star, y)\ge k+1$, and $s_0\in V(\cP_{T'}(s^\star,y))$. 
    \item[(ii)] Let $s^\star\notin \{s_0,s'_0\}$ be a sensor in $\wT_0$. Then, for any $y\in V\setminus\wT_0$ the path $\cP_T(s^\star,y)$ contains $s_0$ or $s'_0$.
\end{itemize}
\end{claim}

In order to prove part of Claim \ref{claim:nocomm-B}, we will first prove the following structural property. Recall the type of vertices with respect to two sensors from Definition \ref{defn:type-height}.

\begin{claim}[Location of sensors]\label{claim:loc}
Consider the notation of Definition \ref{defn:trB}. Then in $T$ for any sensor $s^\star\in S$ one of the following three possibilities holds:
\begin{itemize}
\item[(i)] $\typ_{s_0,s'_0}(s^\star)=0$,
\item[(ii)] $\typ_{s_0,s'_0}(s^\star)=d_T(s_0,s'_0)$ or
\item[(iii)] $\typ_{s_0,s'_0}(s^\star)=q$ and $\{w_q,u_{q+1}\}\in E(\cP_T(s_0,s^\star))$.
\end{itemize}
\end{claim}

\begin{proof}
First, we prove that $\typ_{s_0,s'_0}(s^\star)\notin\{0,q,d_T(s_0,s'_0)\}$ cannot hold. Assume indirectly that it does hold, and $1\le \typ_{s_0,s'_0}(s^\star)\le q-1$. We can assume that $\hgt_{s_0,s'_0}(s^\star)$ is minimal among the sensors of the same type, hence, there is no sensor besides $s^\star$ on the path $\cP_T(s^\star,w_{\typ(s^\star)})$. By the comments after Definition \ref{defn:trB} we have $w_{-1}:=w_{d_T(s_0,s'_0)-1}\in A_T(s_0,s'_0)$. This implies that $s^\star$ cannot measure $w_{-1}$, hence
\begin{equation}\label{eq:loc1}
d_T(s^\star,w_{-1})\ge k+1>d_T(s_0,w_{-1}).
\end{equation}
As $w_q$ lies on both paths $\cP_T(s_0,w_{-1})$ and $\cP_T(s^\star,w_{-1})$, \eqref{eq:loc1} gives
\begin{equation}\label{eq:loc2}
d_T(s^\star,w_q)>d_T(s_0,w_q).
\end{equation}
Furthermore, as $w_q$ lies on both paths $\cP_T(s_0,s_1)$ and $\cP_T(s^\star,s_1)$, \eqref{eq:loc2} implies that $\cP_T(s^\star,s_1)$ is a longer sensor path than $\cP_T(s_0,s_1)$, contradicting Condition \ref{cond:trB} (iii). Therefore, $1\le \typ_{s_0,s'_0}(s^\star)\le q-1$ cannot hold for any sensor $s^\star$.

The proof that $q+1\le \typ_{s_0,s'_0}(s^\star)\le d_T(s_0,s'_0)-1$ cannot hold either for any sensor $s^\star$ is analogous to the above, reversing the roles of $s_0$ and $s'_0$, and those of $w_1$ and $w_{-1}$.

We are left to prove the fact that if $\typ_{s_0,s'_0}(s^\star)=q$, then $\{w_q,u_{q+1}\}\in E(\cP_T(s_0,s^\star))$. Assume to the contrary that there exists a sensor $s^\star$ with $\typ_{s_0,s'_0}(s^\star)=q$, and $\{w_q,u_{q+1}\}\notin E(\cP_T(s_0,s^\star))$, moreover, $d_T(w_q,s^\star)$ is minimal among these sensors. Now if $d_T(w_q,s^\star)\le q$, then
\[
d_T(s^\star,s'_0)=d_T(s^\star,w_q)+d_T(w_q,s'_0)\le
d_T(s_0,w_q)+d_T(w_q,s'_0)=d_T(s_0,s'_0),
\]
implying that $\cP_T(s^\star,s'_0)$ is a strong sensor path sharing an edge with $\cP_T(s_0,s'_0)$, contradicting Condition \ref{cond:trB} (ii). On the other hand, if $d_T(w_q,s^\star)>q$, then
\begin{equation}\label{eq:loc3}
d_T(s^\star,s_1)=d_T(s^\star,w_q)+d_T(w_q,s_1)>
d_T(s_0,w_q)+d_T(w_q,s_1)=d_T(s_0,s_1),
\end{equation}
as $\{w_q,u_{q+1}\}\in E(\cP_T(s_0,s_1))$ and $\{w_q,u_{q+1}\}\notin E(\cP_T(s_0,s^\star))$ by the assumptions, implying that $\{w_q,u_{q+1}\}\in E(s^\star,s_1)$. The consequence of \eqref{eq:loc3} is that $\cP_T(s^\star,s_1)$ is a longer sensor path than $\cP_T(s_0,s_1)$, contradicting Condition \ref{cond:trB}. 
\end{proof}

Notice that Claim \ref{claim:loc} implies, in particular, that for any $v\in V_1$ in Definition \ref{defn:trB}, $\{w_q,u_{q+1}\}\in E(\cP_T(s_0,v))$. In fact,
\begin{equation}\label{eq:V1}
V_1=\{v\in A_T(s_0,s'_0): u_{q+1}\in V(\cP_T(w_q,v))\},
\end{equation}
and thus $A_T(s_0,s'_0)\setminus V_1\subseteq\wT_0$.

\begin{proof}[Proof of Claim \ref{claim:nocomm-B}]
(i) Since $s^\star$ and $y$ are in different components ($\wT_i$ and $\wT_j$), $s_0\in V(\cP_{T'}(s^\star,y))$ follows from the construction of the edge set $E_3$. Now consider the paths $\cP_{T'}(s^\star,s_0)$ and $\cP_{T'}(y,s_0)$. By the construction of the edge set $E_3$ again, the vertices on these two paths neighboring $s_0$ are $x_i$ and $x_j$, respectively. Since $x_i\ne x_j$, at least one of these two is not equal to $u_{q+1}$. Without loss of generality, assume that $x_i\ne u_{q+1}$ (the proof of the other case is analogous). This implies that on the path $\cP_T(x_i,s_0)$ the edge incident to $x_i$ was removed as part of $E_1$ because its other endpoint, say $v$, belonged to $V_1$ (and not because it was the edge $\{w_q,u_{q+1}\}$). Hence, $v\in A_T(s_0,s'_0)$, implying that $d_T(s^\star,v)\ge k+1$. Consequently,
\[
d_T(s^\star,x_i)\ge d_T(s^\star,v)-1\ge k+1-1=k,
\]
and thus
\[
d_{T'}(s^\star,y)\ge d_{T'}(s^\star,s_0)=d_T(s^\star,x_i)+1\ge k+1,
\]
since the path $\cP_T(s^\star,x_i)$ remains untouched by the transformation.

To prove part (ii) notice that by Definition \ref{defn:trB}, $\typ_{s_0,s'_0}(y)=q$. On the other hand, Claim \ref{claim:loc} implies that $\typ_{s_0,s'_0}(s^\star)$ is either $0$ or $d_T(s_0,s'_0)$ (as option (iii) of Claim \ref{claim:loc} would contradict with $s^\star\in\wT_0$). In the former case $\cP_T(s^\star,y)$ contains $\cP_T(s_0,w_q)$ as a sub-path, and in the latter it contains $\cP_T(s'_0,w_q)$ as a sub-path, finishing the proof.
\end{proof}

We continue by showing that there cannot be too many vertices in $V_1\cup V_2$.

\begin{claim}\label{claim:typq}
Consider the notation of Definition \ref{defn:trB}. Then the following hold in $T$:
\begin{itemize}
\item[(i)] for every $h\in\{1,2,\ldots,k-q\}$, $|\{v\in V_1\cup V_2:\ \hgt_{s_0,s'_0}(v)=h\}|\in\{0,1\}$, and
\item[(ii)] for every $h>k-q$, $|\{v\in V_1\cup V_2:\ \hgt_{s_0,s'_0}(v)=h\}|=0$.
\end{itemize}
Consequently, $|V_1\cup V_2|\le k-q$, and every vertex in $V_1\cup V_2$ is directly measured by both $s_0$ and $s'_0$ in $T'=\tr_B(T,S,s_0,s_1)$.
\end{claim}
\begin{proof}
In the following, 'type' and 'height' will always refer to type and height in $T$ with respect to $s_0,s'_0$. First, there cannot be two distinct vertices $x,y\in V_1$ with the same height, since both $x$ and $y$ are of type $q$, and both belong to $A_T(s_0,s'_0)$ (see the discussion before Definition \ref{defn:trB}). Next, by Claim \ref{claim:loc}, for every sensor $s\in S$ and for every $v\in V_2$ it holds that $w_q\in\ V(\cP_T(s,v))$. Hence, for any $h\ge 1$ if there were at least two vertices in $V_2$ with the same height (hence, the same distance from $w_q$), then they would not be distinguished from each other in $T$ by any sensor.

Note that we must have $q=d_T(s_0,w_q)\ge d_T(s'_0,w_q)$ as otherwise $\cP_T(s'_0,s_1)$ would be a longer sensor path than $\cP_T(s_0,s_1)$, violating Condition \ref{cond:trB}(iii). This implies, by the discussion before Definition \ref{defn:trB}, that the maximal height of a type-$q$ vertex in $A_T(s_0,s'_0)$ is $k-q$. Consequently, for any $v\in V_1$, $\hgt_{s_0,s'_0}(v)\le k-q$. To show the same for the vertices of $V_2$ fix some $v\in V_2$ and assume that there exists a sensor $s\in S\setminus\{s_0,s'_0\}$ that directly measures $v$. Then $s$ also directly measures $w_q$ (by Claim \ref{claim:loc}). If $d_T(s,w_q)\le q$ held, then $\cP_T(s,s'_0)$ would be a sensor path that is at most as long as $\cP_T(s_0,s'_0)$, meaning that it would be a strong sensor path. $\cP_T(s,s'_0)$ and $\cP_T(s_0,s'_0)$ would then form a pair of strong sensor paths sharing an edge, contradicting Condition \ref{cond:trB}(ii). Hence, $d_T(s,w_q)>q=\max\{d_T(s_0,w_q),d_T(s'_0,w_q)\}$. Since $w_q\in V(\cP_T(s,v))$, this implies that if $s$ directly measures $v$, then so do both $s_0$ and $s'_0$. Hence, regardless of the locations of the other sensors, $s_0$ and $s'_0$ both measure directly every vertex in $V_2$, implying that their height can be at most $k-q$.

Finally, we have to show that if $x\in V_1$ and $y\in V_2$, then $\hgt_{s_0,s'_0}(x)=\hgt_{s_0,s'_0}(y)$ cannot hold. Assume that it does hold. Then $x\in A_T(s_0,s'_0)$ implies that $y\notin A_T(s_0,s'_0)$ (otherwise they would not be distinguished). That is, there exists a sensor $s_y\in S\setminus \{s_0,s'_0\}$ that directly measures $y$. As before, $w_q \in V(\cP_T(s_y,y))$ holds by Claim \ref{claim:loc}. This implies that
\begin{equation}\label{eq:typq1}
 d_T(s_y,x)\le d_T(s_y,w_q)+d_T(w_q,x)
=d_T(s_y,w_q)+d_T(w_q,y)
=d_T(s_y,y)\le k,  
\end{equation}

where in \eqref{eq:typq1} we used that $\hgt_{s_0,s'_0}(x)=\hgt_{s_0,s'_0}(y)$, combined with $\typ_{s_0,s'_0}(x)=\typ_{s_0,s'_0}(y)=q$. Hence, $s_y$ also measures $x$. By $s_0,s'_0\notin V(\cP_T(s_y,x))$, this means that $s_y$ either directly measures $x$, or there is a sensor $s'\in S\setminus\{s_0,s'_0\}$ on the path $\cP_T(s_y,x)$ that directly measures $x$, contradicting the assumption that $x\in A_T(s_0,s'_0)$.

Combining all the above finishes the proof of (i) and (ii). It then follows that $|V_1\cup V_2|\le k-q$. Since $q=d_{T'}(s_0,w_q)\ge d_{T'}(s'_0,w_q)$, and the vertices of $V_1\cup V_2$ form a single leaf-path emanating from $w_q$ in $T'$, this implies that both $s_0$ and $s'_0$ directly measure every vertex in $V_1\cup V_2$ in $T'$.
\end{proof}

\subsection{Proof that Transformation B works}\label{sec:B-proof}
\begin{proof}[Proof of Lemma \ref{lem:trBprop}]
{\bf Proof of (i):} Let $x, y\in V$ be a pair of distinct vertices. We shall prove that there is a sensor in $S$ that resolves them in $T'$, similarly to the proof of Lemma \ref{lem:trAprop}(i). We will use the notations of Definition \ref{defn:trB}. We will do a case-distinction analysis with respect to the location of $x$ and $y$ in the components $\wT_i, i\ge 0$ or in $V_1$ described in the transformation. The numbering of the cases is consistent with that in the proof of Lemma \ref{lem:trAprop}(i).

{\bf Case 1:} Assume that $x\in\wT_i$ for some $i\ge 1$, and that $y\in\wT_j$ for some $j\ge 0$, $j\ne i$. Then, since $x\in V\setminus(A_T(s_0)\cup A_T(s_0')\cup A_T(s_0,s_0'))$, there is a sensor $s'\in S\setminus\{s_0,s_0'\}$ that directly measures $x$. Then, by Claim \ref{claim:nocomm-B}(i)--(ii), $s'\in\wT_i$. Therefore, the edges of $\cP_T(s',x)$ are unchanged in $T'$, so $s'$ still directly measures $x$ in $T'$. Then, either $s'$ distinguishes $x$ and $y$ in $T'$, or
\begin{equation}\label{eq:trB1}
d_{T'}(s',y)=d_{T'}(s',x)\le k.
\end{equation}
Assume that we have this latter case. Now Claim \ref{claim:nocomm-B}(i) implies that $s_0\in V(\cP_{T'}(s',y))$ (this is also true if $y\in\wT_0$ by the construction). Then $d_{T'}(s_0,y)\le d_{T'}(s',y)\le k$, so $s_0$ also measures $y$ in $T'$. We will prove that in this case $s_0$ distinguishes $x$ and $y$ in $T'$. Assume that this is not the case, and in fact
\begin{equation}\label{eq:trB2}
d_{T'}(s_0,x)=d_{T'}(s_0,y)\le k.
\end{equation}
Then \eqref{eq:trB1} and \eqref{eq:trB2} imply that
\begin{equation}\label{eq:trB3}
d_{T'}(s',y)-d_{T'}(s_0,y)=d_{T'}(s',x)-d_{T'}(s_0,x).
\end{equation}
Since $s_0\in V(\cP_{T'}(s',y))$, the left-hand side of \eqref{eq:trB3} is equal to $d_{T'}(s',s_0)$. However, this could only be equal to the right-hand side of \eqref{eq:trB3} if $s_0\in V(\cP_{T'}(s',x))$ held, which cannot be the case, since $\cP_{T'}(s',x)$ is fully contained in $\wT_i$. This contradiction finishes the proof of Case 1.

{\bf Case 2:} Now assume that $x,y\in\wT_i$ for some $i\ge 1$. Let $s'$ be a sensor that distinguishes them in $T$. If $s'\in\wT_i$, then the paths $\cP_T(s',x)$ and $\cP_T(s',y)$ remain unchanged by the transformation, hence $s'$ still distinguishes $x$ and $y$ in $T'$ and we are done. If $s'\notin\wT_i$, then $s'\in\wT_0$ by Claim \ref{claim:nocomm-B}(i). Then, since $x,y\in\wT_i$, the paths $\cP_T(s',x)$ and $\cP_T(s',y)$ both contain $x_i$ (the closest vertex to $w_q$ in $\wT_i$). On the other hand, the paths $\cP_T(x_i,x)$ and $\cP_T(x_i,y)$ fully belong to $\wT_i$, so they are unchanged by the transformation, while the edge $\{s_0,x_i\}$ is added when creating $T'$. Combining these facts gives
\begin{align}
d_{T'}(s_0,x)&=1+d_{T'}(x_i,x)=1+d_T(x_i,x)\le d_T(s',x_i)+d_T(x_i,x)=d_T(s',x),\label{eq:trBcase21}\\
d_{T'}(s_0,y)&=1+d_{T'}(x_i,y)=1+d_T(x_i,y)\le d_T(s',x_i)+d_T(x_i,y)=d_T(s',y),\label{eq:trBcase22}
\end{align}
and we assumed that the rhs of both sides is at most $k$, so $s_0$ measures both $x$ and $y$ in $T'$, and further, since we replaced the segment $\cP_T(s_0, x_i)$ by a single edge in $T'$,
\begin{equation}\label{eq:trBcase23}
d_{T'}(s_0,x)-d_{T'}(s_0,y)=d_T(x_i,x)-d_T(x_i,y)=d_T(s',x)-d_T(s',y).
\end{equation}
The combination of \eqref{eq:trBcase21}, \eqref{eq:trBcase22} and \eqref{eq:trBcase23} implies that if $s'$ distinguished $x$ and $y$ in $T$, then $s_0$ distinguishes them in $T'$, finishing the proof.

{\bf Case 3:} Next, assume that $x,y\in\wT_0$. Then there is a sensor $s'$ that resolves them in $T$. If $s'\in\wT_0$, then the paths $\cP_T(s',x)$ and $\cP_T(s',y)$ completely lie in $\wT_0$, so they stay intact during the transformation. Hence, $s'$ still resolves $x,y$ in $T'$ and we are done. Now assume that $s'\in\wT_i$ for some $i\ge 1$. We will show that in this case either $s_0$ or $s'_0$ resolves $x,y$ in $T'$.

By Claim \ref{claim:loc}, $\typ_{s_0,s'_0}(s')=q$ in $T$, and $\{w_q,u_{q+1}\}\in E(\cP_T(w_q,s'))$. Hence, both $\cP_T(s',x)$ and $\cP_T(s',y)$ contain $w_q$. Without loss of generality we can assume that there is no sensor besides $s'$ on the path $\cP_T(s',w_q)$ (if there was one, we could relabel that to $s'$). Now if $d_T(s',w_q)\le\max\{d_T(s_0,w_q),d_T(s'_0,w_q)\}$ then two of $\cP_T(s_0,s'_0)$, $\cP_T(s_0,s')$ and $\cP_T(s'_0,s')$ would form a pair strong sensor paths that share an edge, contradicting Condition \ref{cond:trB}(ii). Hence, $d_T(s',w_q)>\max\{d_T(s_0,w_q),d_T(s'_0,w_q)\}$, thus
\begin{equation}\label{eq:also-measure-1}
    d_T(s_0,x)\le d_T(s_0,w_q)+d_T(w_q,x)<d_T(s',w_q)+d_T(w_q,x)=d_T(s',x),
\end{equation}

and the same holds for $s'_0$ in place of $s_0$. This implies that if $s'$ measures $x$ in $T$, then so do $s_0$ and $s'_0$, and the same reasoning holds for $y$.

We know that $s'$ measures at least one of $x$ and $y$ in $T$, say $x$. Then, by \eqref{eq:also-measure-1} both $s_0$ and $s'_0$ also measure $x$ in $T$. Since $s_0,s'_0,x,y\in\wT_0$, the paths $\cP_T(s_0,x)$, $\cP_T(s'_0,x)$ remain unchanged by the transformation, hence, $s_0$ and $s'_0$ both measure $x$ in $T'$ too. Now assume that neither $s_0$ nor $s'_0$ distinguishes $x,y$ in $T'$.  Recalling Definition \ref{defn:type-height}, this means that
\begin{equation*}
   \typ_{s_0,s'_0}(x)=\typ_{s_0,s'_0}(y),\qquad
\hgt_{s_0,s'_0}(x)=\hgt_{s_0,s'_0}(y),
\end{equation*}
in $T'$, and thus also in $T$, otherwise at least one of $s_0, s_0'$ would resolve $x,y$. In particular this also implies that $d_T(w_q,x)=d_T(w_q,y)$. Then
\begin{align*}
d_{T}(s',x)&=d_{T}(s',w_q)+d_{T}(w_q,x)
=d_T(s',w_q)+d_T(w_q,y)
=d_T(s',y)
\end{align*}
by $w_q\in V(\cP_T(s',x))\cap V(\cP_T(s',y))$. This contradicts the assumption that $s'$ resolved $x,y$ in $T$, thus finishing the proof.

{\bf Case 4:} Assume that $x,y\in V_1\cup V_2$. Then, by the construction of the edge set $E_2$, $x$ and $y$ will both lie on a leaf-path in $T'$ emanating from $w_q$. Hence, $\typ_{s_0,s'_0}(x)=\typ_{s_0,s'_0}(y)$ and $\hgt_{s_0,s'_0}(x)\ne\hgt_{s_0,s'_0}(y)$ in $T'$, so $x$ and $y$ are distinguished by $s_0$ or $s'_0$ (or both) as long as at least one of them is measured by at least one of $s_0$ and $s'_0$. But in fact both $x$ and $y$ are measured by both $s_0$ and $s'_0$ in $T'$ by Claim \ref{claim:typq}, finishing the proof.

{\bf Case 5:} Next, assume that $x\in\wT_i$ for some $i\ge 1$, and $y\in V_1\cup V_2$. The proof in this case is identical to that of Case 1.

{\bf Case 6:} Finally, assume that $x\in\wT_0$ and $y\in V_1\cup V_2$. Then, by the construction of $T'$, and by Claim \ref{claim:typq}, in $T'$ we have $\typ_{s_0,s'_0}(y)=q$ and $1\le\hgt_{s_0,s'_0}(y)\le k-q$, while for $x$ either $x=w_q$ or $\typ_{s_0,s'_0}(x)\ne q$. This implies that $x,y$ are resolved by $\{s_0,s'_0\}$ in $T'$, either by Claim \ref{claim:type-difference} or by the fact that the height is not identical.

{\bf Proof of (ii):} Assume that for a sensor $s\in S$, the vertices of $A_T(s)$ form a leaf-path emanating from $s$ in $T$. Then none of the vertices in $A_T(s)$ are adjacent to a vertex in $V_1\subseteq A_T(s_0,s_0')$. Also, any vertex in $A_T(s)$ can only be adjacent to a type-$q$ vertex in $T$ (with respect to $s_0,s'_0$) if $\typ_{s_0,s'_0}(s)=q$. In this case, by Claim \ref{claim:loc}, $\{w_q,u_{q+1}\}\in E(\cP_T(w_q,s))\subseteq E(\cP_T(w_q,v))$ for any $v\in A_T(s)$. Thus, none of the vertices in $A_T(s)$ are adjacent to a vertex in $V_2$ either. Hence, the removal of the edge set $E_1$ does not change this leaf-path. The addition of the edge set $E_2$ also does not add an edge adjacent to any of the vertices in $A_T(s)$. After these steps, if any one of the vertices in $A_T(s)\cup \{s\}$ is in $\wT_i$ for some $i\ge 1$, then all of $A_T(s)\cup \{s\}$ are in $\wT_i$, and $s$ was closer to $s_0$ in $T$ than any vertex in $A_T(s)$. Hence, the addition of the new edge between $s_0$ and $\wT_i$ in $E_3$ will again not add an edge adjacent to any vertex in $A_T(s)$. This proves that $A_T(s)\subseteq A_{T'}(s)$, and the vertices of $A_T(s)$ still form a leaf-path emanating from $s$ in $T'$.

Next, we have to show that there is no new vertex in $A_{T'}(s)$ compared to $A_T(s)$ for any sensor $s$. Assume that $x\in V\setminus(S\cup(\cup_{s\in S}A_T(s)))$. To prove that $x\notin A_{T'}(s)$ for any $s\in S$, we distinguish the following cases.

{\bf Case 1:} Assume that $\typ_{s_0,s'_0}(x)=0$ in $T$. Then any sensor $s^\star\in S$ that directly measures $x$ in $T$ also has $\typ_{s_0,s'_0}(s^\star)=0$ (otherwise $s_0\in V(\cP_T(s^\star,x))$ would hold). Hence, the path $\cP_T(s^\star,x)$ remains unchanged by the transformation. Since $x\in V\setminus(S\cup(\cup_{s\in S}A_T(s)))$, there are at least two sensors $s^\star_1,s^\star_2$ that directly measure $x$ in $T$, hence, by the above reasoning applied twice, they both measure $x$ directly in $T'$ too, proving that $x\notin A_{T'}(s)$ for any $s\in S$.

{\bf Case 2:}  Assume that $\typ_{s_0,s'_0}(x)=d_T(s_0,s'_0)$ in $T$. The proof in this case is identical to that of Case 1 with $s'_0$ taking the role of $s_0$, and type $d_T(s_0,s'_0)$ taking that of type 0.

{\bf Case 3:} Assume that $\typ_{s_0,s'_0}(x)\notin\{0,q,d_T(s_0,s'_0)\}$ in $T$. Then by Definition \ref{defn:trB} the paths $\cP_T(s_0,x)$, $\cP_T(s'_0,x)$ remain untouched by the transformation. We will show that this implies that both $s_0$ and $s'_0$ directly measure $x$ in $T'$, and as a result $x\notin\cup_{s\in S}A_{T'}(s)$.
Assume that $x\notin \cup_{s\in S}A_{T'}(s)\cup S$. Then there is at least two sensors $s_1^\star,s_2^\star$ directly measuring $x$ in $T$. We shall show that $s_0, s_0'$ can have these roles.
If we immediately know that $s_0, s_0'$ both directly measure $x$, we are done. Suppose now that we only know that there is an $s^\star\in S\setminus \{s_0, s_0'\}$ that directly measures $x$. 

 By Claim \ref{claim:loc}, for every sensor $s\in S$, $\typ_{s_0,s'_0}(s)\in\{0,q,d_T(s_0,s'_0)\}$. Since $w_q\notin S$, and $\typ_{s_0,s'_0}(x)\neq \{0,q,d_T(s_0,s'_0)\}$ by assumption, this means that there is no sensor on the paths $\cP_T(s_0,x)$, $\cP_T(s'_0,x)$ besides $s_0$, $s'_0$, respectively. Furthermore, for every sensor $s^\star\in S$ that directly measures $x$, $w_q\in V(\cP_T(s^\star,x))$, so $s^\star$ also directly measures $w_q$. Now we use the fact that in this case $d_T(s^\star,w_q)>\max\{d_T(s_0,w_q),d_T(s'_0,w_q)\}$, as in the proof of Claim \ref{claim:typq}, since otherwise two of the paths $\cP_T(s^\star,s_0)$, $\cP_T(s^\star,s'_0)$ and $\cP_T(s_0,s'_0)$ would form a pair of strong sensor paths sharing an edge, contradicting Condition \ref{cond:trB}. Hence,
\[
d_T(s_0,x)\le d_T(s_0,w_q)+d_T(w_q,x)<d_T(s^\star,w_q)+d_T(w_q,x)=d_T(s^\star,x)\le k,
\]
and the same holds for $s'_0$ in place of $s_0$. Therefore, both $s_0$ and $s'_0$ directly measure $x$ in $T$. By the fact that $\mathrm{typ}_{s_0,s_0'}(x)\neq q$, the paths $\cP_T(s_0, x), \cP_T(s_0', x)$ are untouched by transformation B, $s_0, s_0'$ both directly measure $x$ in $T'$ as well, showing that $x\notin A_{T'}(s)$ for any sensor $s\in S$.

{\bf Case 4:} Assume that $x\in V_1\cup V_2$. Then $x$ will lie on a leaf-path in $T'$ emanating from $w_q$, which is of length at most $k-q$ by Claim \ref{claim:typq}. Hence, $x$ will be directly measured by both $s_0$ and $s'_0$ in $T'$, implying that $x\notin A_{T'}(s)$ for any $s\in S$.

{\bf Case 5:} The last case is when $\typ_{s_0,s'_0}(x)=q$ in $T$ and $x\notin V_1\cup V_2$. Then, by Definition \ref{defn:trB}, $x\in\wT_i$ for some $i\ge 1$, and $x\notin A_T(s_0,s'_0)$. This latter fact implies that there exist two distinct sensors $s_1^\star,s_2^\star$ that directly measure $x$ in $T$ and $\{s_1^\star,s_2^\star\}\ne\{s_0,s'_0\}$. If both $s_1^\star,s_2^\star\in\wT_i$, then the paths $\cP_T(s_1^\star,x)$, $\cP_T(s_2^\star,x)$ remain unchanged by the transformation, hence, both $s_1^\star$ and $s_2^\star$ still directly measure $x$ in $T'$, finishing the proof. If either $s_1^\star$ or $s_2^\star$ is not in $\wT_i$, then it has to be in $\{s_0,s'_0\}$ by Claim \ref{claim:nocomm-B}(i)--(ii). Since $\{s_1^\star,s_2^\star\}\ne\{s_0,s'_0\}$, we can assume in this case that $s_1^\star\in\{s_0,s'_0\}$ and $s_2^\star\in\wT_i$. Similarly as above, $s_2^\star$ then directly measures $x$ in $T'$. We will finish the proof by showing that the fact that either $s_0$ or $s'_0$ directly measures $x$ in $T$ implies that $s_0$ directly measures $x$ in $T'$. Since $x_i$ is the closest vertex of $\wT_i$ to $s_0$ in $T$, and the edge $\{s_0,x_i\}$ is added in $T'$ by the transformation, we have
\[
d_{T'}(s_0,x)=1+d_{T'}(x_i,x)=1+d_T(x_i,x),
\]
since the path $\cP_T(x_i,x)$ remains unchanged by the transformation. Hence,
\[
d_{T'}(s_0,x)\le\min\{d_T(s_0,x_i),d_T(s'_0,x_i)\}+d_T(x_i,x)=\min\{d_T(s_0,x),d_T(s'_0,x)\}\le k,
\]
as $x_i$ lies on both paths $\cP_T(s_0,x)$ and $\cP_T(s'_0,x)$. This shows that $s_0$ indeed measures $x$ in $T'$. The fact that $s_0$ directly measures $x$ in $T'$ follows from the fact that $\cP_{T'}(x_i,x)$ does not contain any sensors, since it is a subpath of both $\cP_{T}(s_0,x)$ and $\cP_T(s'_0,x)$, and one of these did not contain any internal sensors by the assumption. This finishes the proof that there are indeed at least two sensors that directly measure $x$ in $T'$, and thus $x\notin A_{T'}(s)$ for any $s\in S$.

{\bf Proof of (iii):} If $\cP_{T'}(s,s')$ is a sensor path, then there are two possible cases. First, if $s,s'\in\wT_i$ for some $i\ge 0$, then $\cP_{T'}(s,s')$ is the same as $\cP_T(s,s')$ (its edges remain unchanged by the transformation), hence, $\cP_T(s,s')$ is also a sensor path with the same length. Second, if, say, $s'=s_0$, and $s\in\wT_i$ for some $i\ge 1$, then the path $\cP_{T'}(s_0,s)$ consists of the sub-path $\cP_T(x_i,s)$ of $\cP_T(s_0,s)$ and the edge $\{s_0,x_i\}$, where recall that $x_i$ is the closest vertex of $\wT_i$ to $s_0$ in $T$. If $\cP_{T'}(s_0,s)$ is a sensor path, then there is no sensor beside $s$ on $\cP_{T'}(x_i,s)=\cP_{T}(x_i,s)$. On the other hand, if $y$ is the vertex neighboring $x_i$ on the path $\cP_T(s_0,x_i)$, then either $y\in A_T(s_0,s_0')$ or $y=w_q$, both implying that $s_0$ directly measures $y$ in $T$, that is, there is no other sensor between them. Consequently, $\cP_T(s_0,s)$ is indeed a sensor path, and its length is more than that of $\cP_{T'}(s_0,s)$, as $y\in V(\cP_T(s_0,s))\setminus V(\cP_{T'}(s_0,s))$.

There are indeed no more cases for a sensor path $\cP_{T'}(s,s')$, as the construction in Definition \ref{defn:trB} ensures that if $s$ and $s'$ are in different components among $\wT_0$, $\wT_1,\ldots$, then $\cP_{T'}(s,s')$ contains $s_0$.

{\bf Proof of (iv):} By Definition \ref{defn:trB}, $x_1$ is the closest vertex to $s_0$ in $T$ in the subtree $\wT_1$. Since $T$ is a tree, and $s_1\in\wT_1$, $x_1$ lies on the path $\cP_T(s_0,s_1)$. This implies that $\cP_{T}(x_1,s_1)\subseteq \cP_T(s_0,s_1)$, and the edges of $\cP_{T}(x_1,s_1)$ stay intact in $T'$. In particular, $\cP_{T'}(x_1,s_1)$ does not contain another sensor besides $s_1$. Since $\{s_0,x_1\}$ is an edge in $T'$, this proves that $\cP_{T'}(s_0,s_1)$ is indeed a sensor path in $T'$.

Next, we will prove that $\cP_{T'}(s_0,s_1)$ is strictly shorter than $\cP_{T}(s_0,s_1)$. Since $x_1\in\wT_1$, and $w_1\in\wT_0$, it holds that $x_1\ne w_1$, furthermore, $w_1\in V(\cP_T(s_0,x_1))\setminus\{s_0,x_1\}$. Hence, $|E(\cP_T(s_0,x_1))|\ge 2$, while $\{s_0,x_1\}$ is a single edge in $T'$. Since the edges of $\cP_T(x_1,s_1)$ remain unchanged in $T'$,
\[ \cP_T(s_0,s_1)=\cP_T(s_0, x_1)\cup \cP_T(x_1, s_1) \quad \mbox{ and } \quad  \cP_{T'}(s_0,s_1)=\{s_0,x_1\}\cup \cP_T(x_1, s_1),\] this finishes the proof.

{\bf Proof of (v):} Let $z$ be the vertex next to $s_1$ on the path $\cP_T(s_0,s_1)$. Since $\cP_T(s_0,s_1)$ is a weak sensor path, $s_0$ cannot measure $z$ in $T$. On the other hand, $s_0'$ cannot measure $z$ in $T$ either, as we will now show. Assume $s_0'$ measures $z$ in $T$. Then we have
\begin{equation}\label{eq:trBv1}
d_T(s_0',s_1)\le d_T(s_0',z)+d_T(z,s_1)\le k+1.
\end{equation}
We also have $V(\cP_T(s'_0,s_1))=V(\cP_T(s'_0,w_q))\cup V(\cP_T(w_q,s_1))$, and $V(\cP_T(w_q,s_1))\subseteq V(s_0,s_1)$. Since $\cP_T(s_0,s_1)$ and $\cP_T(s_0,s'_0)$ are both sensor paths, it follows that $\cP_T(s'_0,s_1)$ is also a sensor path. Hence, \eqref{eq:trBv1} shows that $\cP_T(s_0',s_1)$ is a strong sensor path. Thus $\cP_T(s_0,s_0')$ and $\cP_T(s_0',s_1)$ are a pair of strong sensor paths sharing an edge. This contradicts with Condition \ref{cond:trB}(ii), showing that neither $s_0$ nor $s_0'$ can measure $z$ in $T$. But $z\in\cP_T(s_0,s_1)$, so by Condition \ref{cond:trB}(i), $z\notin A_T(s_1)$, i.e., there has to be another sensor besides $s_1$ that directly measures $z$, say $s_2$. This, similarly to \eqref{eq:trB1}, shows that $\cP_T(s_2,s_1)$ is a strong sensor path. This, in particular, implies that $s_2\in\wT_1$, and the edges of $\cP_T(s_2,s_1)$ remain unchanged in $T'$. Now, if after the transformation, $\cP_{T'}(s_0,s_1)$ becomes a strong sensor path, then $s_2\notin\cP_{T'}(s_0,s_1)$ shows that $\cP_{T'}(s_0,s_2)$ and $\cP_{T'}(s_2,s_1)$ form a pair of strong sensor paths that share an edge, thus finishing the proof of part (v).
\end{proof}

\section{Step C: Overlapping short sensor paths are suboptimal}\label{sec:C}
Our final building block will be Lemma \ref{lem:trC} that shows that an optimal tree cannot have overlapping strong sensor paths. We will fix notation related to overlapping strong sensor paths as follows.
To help the reader we try to make notation as similar to Transformations A and B as possible.

\begin{cond}\label{cond:shortest}
Let $T=(V,E)$ be a tree with a threshold-$k$ resolving set $S\subseteq V$. Assume that $T$ has at least one pair of strong sensor paths that share an edge. We assume that 
$\cP_T(s_0, s'_0)$ is one of the shortest among the strong sensor paths that share an edge with another strong sensor path, and that $\cP_T(s_0, s_1)$ is one of the shortest among the strong sensor paths that share an edge with $\cP_T(s_0, s'_0)$.
With $|V(\cP_T(s_0, s'_0))\cap V(\cP_T(s_0, s_1))|=q\le k-1$, we denote the vertices on these two paths as follows:
\begin{equation}\label{eq:s0s1s2}
\begin{aligned}
    V(\cP_T(s_0, s'_0))&=\{s_0, w_1, \dots, w_q, w_{q+1}, \dots, w_{d(s_0, s'_0)-1}, s'_0\},\\ V(\cP_T(s_0, s_1))&=\{s_0, w_1, \dots, w_q, u_{q+1}, \dots, u_{d(s_0, s_1)-1}, s_1\},
\end{aligned}
\end{equation}
with $q\le d_T(s_0, s'_0)-q\le d_T(s_0, s_1)-q$, and each sensor $s^\star\in S\setminus \{s_0, s'_0, s_1\}$ that directly measures $w_q$ has $d_T(s^\star, w_q) \ge d_T(s_0, s_1)-q$.
\end{cond}
The last statement is indeed true since the assumptions that $\cP_T(s_0, s'_0)$ is the shortest strong sensor path among the ones sharing an edge with another strong one, and that $\cP_T(s_0, s_1)$ is at most as long as $\cP_T(s'_0, s_1)$ imply that $d_T(s_0, w_q) \le d_T(s'_0, w_q)\le d_T(s_1, w_q)$. I.e., among the sensors $s\in S$ for which there is no other sensor on the path $\cP_T(s,w_q)$, the closest one to $w_q$ is $s_0$, the second closest one is $s'_0$, and the third closest one is $s_1$ (ties are allowed). Moreover, all three of $s_0,s'_0,s_1$ (directly) measure $w_q$, since $\cP_T(s_0,s'_0)$ and $\cP_T(s_0,s_1)$ are strong sensor paths. 

Note that it is possible that $w_{q+1}=s'_0$ or $u_{q+1}=s_1$, but only if $q=1$.
\begin{lemma}\label{lem:trC}
Let $T=(V,E)$ be a tree with a threshold-$k$ resolving set $S$ where $|S|=m$. Assume that for all $s\in S$, $A_{T}(s)$ is contained in a single leaf-path starting from $s$. Then, if $T$ contains a pair of strong sensor paths $\cP_T(s_0, s'_0)$ and $\cP_T(s_0, s_1)$ that share an edge, then $T\notin \cT_m^\star$.
\end{lemma}

\begin{corollary}\label{cor:trC}
Let $T=(V,E)$ be a tree  with a threshold-$k$ resolving set $S$ where $|S|=m$. Assume that for all $s\in S$, $A_{T}(s)$ is contained in a single leaf-path starting from $s$. If  $(T,S)$ either has a weak sensor path or a pair of strong sensors paths that share an edge,  then $T\notin \cT_m^\star$.
\end{corollary}

\begin{proof}
If  $(T,S)$ has a pair of strong sensors paths that share an edge,  then Lemma \ref{lem:trC} is immediately applicable, yielding that $T\notin \cT_m^\star$.
In case $(T,S)$ has a weak sensor path then notice that the application of Lemma \ref{lem:trB} for $(T,S)$ results in another tree $\widehat{T}$ on the same vertex set, and with $S$ still being a threshold-$k$ resolving set on $\widehat T$, for which Lemma \ref{lem:trC} can be applied, and so again  $T\notin \cT_m^\star$.
\end{proof}

The proof of Lemma \ref{lem:trC} relies on a third kind of edge-rewiring procedure on $T$, that we call {\it Transformation C}. Heuristically speaking, this is what we will do: 
We 'separate' the overlapping paths $\cP_T(s_0,s'_0)$ and $\cP_T(s_0,s_1)$ by keeping the former one intact, while in $\cP_T(s_0,s_1)$ we replace the segment $\cP_T(s_0,u_{q+1})$ by a new path $(s_0, v^\star, u_{q+1})$ with a new vertex $v^\star$, while cutting the edge $\{w_q,u_{q+1}\}$. This way we increase the number of vertices in the graph while not increasing the length of either $\cP_T(s_0,s'_0)$ or $\cP_T(s_0,s_1)$. Now the vertices that were measured in $T$ by $s_0, s_0'$ 'through' $w_q$ might not be distinguished from some other vertices anymore, since we cut the edge $\{w_q,u_{q+1}\}$. To solve this problem we make some further changes in the graph. We \emph{pretend} that $w_q$ is a sensor, but with a smaller measuring radius $k-(d_T(s_0, s'_0)-q)$. This is the distance up to which $s_0'$ measured vertices through $w_q$ in $T$. We then `cut out' the attraction of $w_q$ (vertices that are only directly measured by $w_q$, if it were the above-mentioned sensor)  from the tree, and move it to a leaf-path emanating from $w_q$, similarly to transformation A. We obtain then a forest. Then we connect each connected component of this forest to form a new tree by connecting $s_0$ to the originally closest vertex in every other component, again similarly to transformation A. Formally, the transformation is as follows.
\begin{defn}[Transformation C]\label{defn:trC}
Let $T=(V,E)$ be a tree with a threshold-$k$ resolving set $S\subseteq V$ satisfying Condition \ref{cond:trB}(i), and $s_0, s'_0, s_1\in S$, $w_q\in V$ satisfying the setting of Condition \ref{cond:shortest}.
Let, for some $\ell\in \{0, 1, \dots, k-(d_T(s_0, s'_0)-q)\}$,
\begin{equation}\label{eq:setX}
    A_T^\star(w_q):=\big\{x\in V\setminus S: \forall s^\star \in S: d(s^\star,x) \ge k+1 \mbox{ or } w_q\in V(\cP_T(s^\star,x))\big\}=:\{v_1, \dots, v_\ell\},
\end{equation}
where $d_T(w_q, v_i)\le d_T(w_q, v_j)$ when $i\le j$.
Then define the following edge sets.
\begin{align}
    E_1&:= \{\{w_q, u_{q+1}\}\}\cup \big\{ \{x,y\}\in E(T): x\in A_T^\star(w_q) \mbox{ or }  y\in A_T^\star(w_q)  \big\}\label{eq:trC-e1},\\
    E_2&:=\{\{w_q, v_1\}\}\bigcup (\cup_{i=1}^{\ell-1} \{\{v_i, v_{i+1}\}\})\label{eq:trC-e2}.
\end{align}
Let $\wT_0, \wT_1, \dots, \wT_r$ be the connected components of $\wT=(V\setminus A^\star_T(w_q),E\setminus E_1)$, with $\wT_0$ containing $s_0$ (and the whole path $\cP_T(s_0,s'_0)$), and $\wT_1$ containing $s_1$. Let $x_i$ be the unique closest vertex of $\wT_i$ to $w_q$ in $T$, for $i\in\{2,3,\ldots,r\}$, and let $v^\star$ be a new vertex, which is not in $V$. Then we also define the edge set
\begin{equation}
E_3:=\Big(\{\{s_0,v^\star \}\}\cup \{\{v^\star, u_{q+1}\}\}\Big)\bigcup\Big(\cup_{i=1}^r \{\{s_0, x_i\}\}\Big). \label{eq:trC-e3}
\end{equation}
Then define $\tr_C(T, S, s_0, s'_0, s_1):=(V', E')$ where $V':=V\cup\{v^\star\}$, and $E':=(E\setminus E_1)\cup E_2 \cup E_3 $.
\end{defn}

For an example of Transformation C see Figure \ref{fig:trC}.

\begin{figure}[ht]
    \centering
    \begin{subfigure}[b]{0.49\textwidth}
    \centering
    \includegraphics[width=\textwidth]{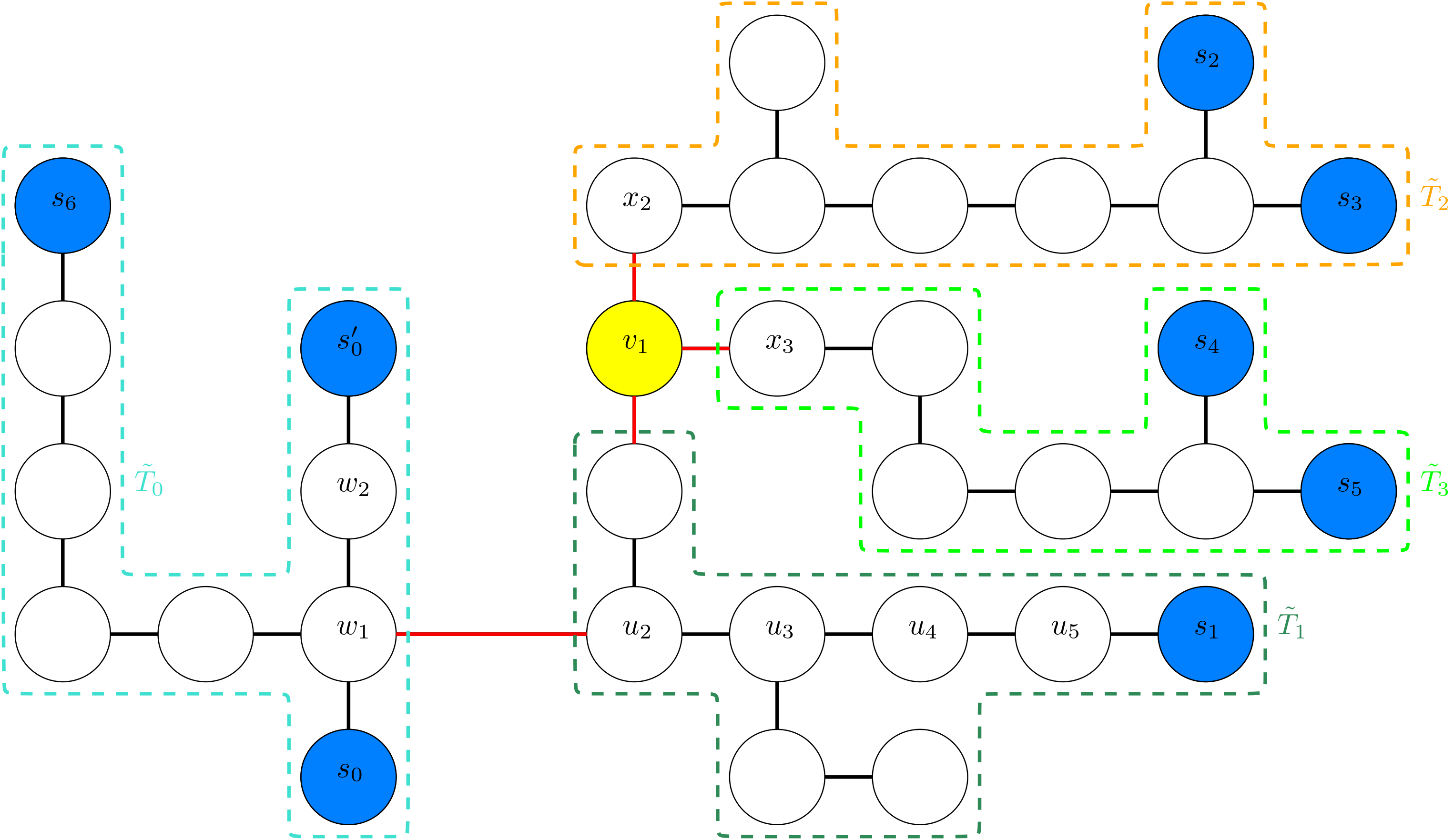}
    \end{subfigure}
    \hfill
    \begin{subfigure}[b]{0.49\textwidth}
    \centering
    \includegraphics[width=\textwidth]{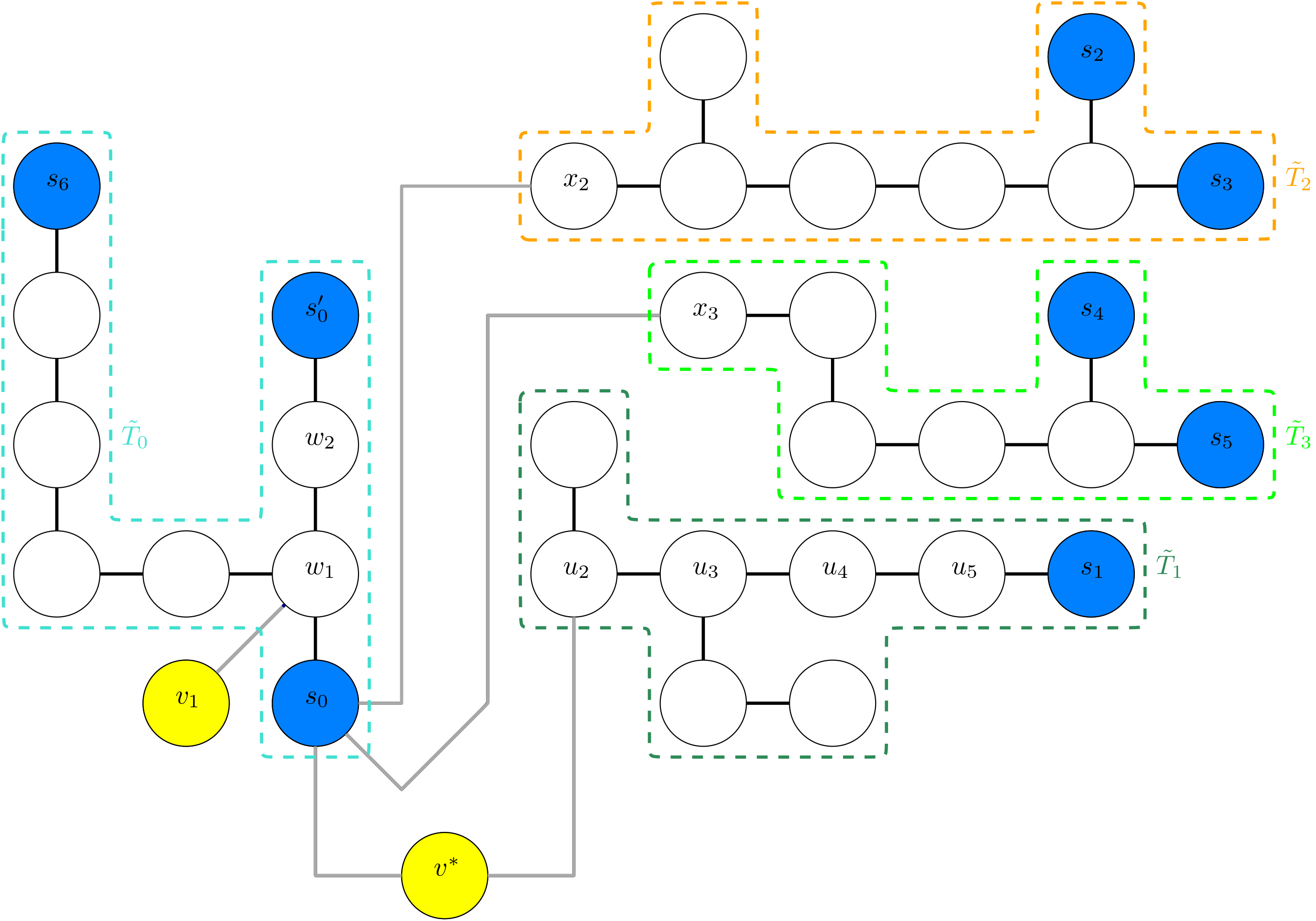}
    \end{subfigure}
    \caption{An example of transformation C with $T$ on the left and $T'=\tr_C(T, S, s_0, s'_0, s_1)$ on the right. Here $k=5$. The blue vertices are the sensors in $S$, and the yellow vertices are in $A^\star_T(w_q)\cup\{v^\star\}$.  The red edges belong to $E_1$, and are deleted by the transformation. The grey edges belong to $E_2\cup E_3$, and are added by the transformation. The subtrees $\wT_0,\wT_1,\wT_2$ and $\wT_3$ are also highlighted. Observe that vertex $x_2$ and the leaf vertex at distance two from $x_2$ in $\widetilde T_2$ are only resolved by $s_0$ both in $T$ and in $T'$. In $T$, the sensors in $\widetilde T_0$ may measure vertices in any subtree, while $s_1$ may only measure vertices in $\widetilde T_1$ and possibly in $\widetilde T_0$. Sensors in $\wT_2$ and $\wT_3$ can only measure vertices within their own subtrees.}
    \label{fig:trC}
\end{figure}

We make a couple of comments on this definition. Since $\cP_T(s_0,s'_0)$ and $\cP_T(s_0,s_1)$ are both strong sensor paths, all of their vertices are measured directly by both endpoints of the path. Hence, $V(\cP_T(s_0,s'_0))\cup V(\cP_T(s_0,s_1))\subseteq V\setminus A_T^\star(w_q)$, and in fact, $V(\cP_T(s_0,s'_0))\subseteq \wT_0$, and $V(\cP_T(u_{q+1},s_1))\subseteq\wT_1$. (Note that $\wT_1$ is the only component $\wT_i$ that was not separated from $\wT_0$ by a vertex in $A^\star_T(w_q)$, but by the additional cut that we made at the edge $\{w_q,u_{q+1}\}$.) Next, $|A^\star_T(w_q)|\le k-(d_T(s_0,s'_0)-q)$, as we will now show. By the definition of $A^\star_T$, for every sensor $s$ that measures a vertex $x\in A^\star_T$, $w_q\in V(\cP_T(s,x))$. Hence, if $x,y\in A^\star_T(w_q)$ had the same distance from $w_q$, then they would have the same distance from every sensor that measures at least one of them, contradicting the fact that $S$ is a threshold resolving set in $T$. On the other hand, the largest distance a vertex in $A^\star_T$ can be from $w_q$ is $k-(d_T(s_0,s'_0)-q)$, otherwise $s'_0$ would not measure it, meaning that only $s_0$ could measure it directly (by the remarks after Condition \ref{cond:shortest}), contradicting Condition \ref{cond:trB}(i).
\subsection{Properties of Transformation C and their consequences}\label{sec:C-properties}
\begin{lemma}\label{lem:trCprop}
Let $T=(V,E)$ be a tree with a threshold-$k$ resolving set $S\subseteq V$ such that Condition \ref{cond:trB}(i) and Condition \ref{cond:shortest} hold. Then $S$ is still a threshold-$k$ resolving set in $\tr_C(T,S,s_0,s'_0,s_1)$.
\end{lemma}
\begin{proof}[Proof of Lemma \ref{lem:trC} subject to Lemma \ref{lem:trCprop}]
Suppose that $T$ satisfies Condition \ref{cond:trB}(i) and contains a  pair of strong sensor paths that share an edge. Then by choosing the shortest among the sensor paths that share an edge with another sensor path, and then choosing the shortest among those that overlap with the first path  we can identify $s_0, s_0', s_1$ and assume that Condition \ref{cond:shortest} holds. From now on we will use the notation therein. In this case, Lemma \ref{lem:trCprop} implies that $S$ is a threshold-$k$ resolving set in $T'=\tr_C(T,S,s_0,s'_0,s_1)$, whereas $T'$ has one more vertex than $T$, proving that $T\notin\mathcal{T}_m^\star$.
\end{proof}
\subsection{Preliminaries to treating Transformation C}\label{sec:C-prelim}
In order to prove Lemma \ref{lem:trCprop} we will make use of the following Claim.
\begin{claim}[No 'communication' between different subtrees]\label{claim:nocomm-C}
Consider the setting and notation of Definition \ref{defn:trC} and let $T'=\tr_C(T,S,s_0,s'_0,s_1)$.
\begin{itemize}
\item[(i)]
Let $s^\star$ be any sensor in $\wT_i$ for some $i\in\{2,3,\ldots,r\}$. Then, for all vertices $y\notin\wT_i$, $d_T(s^\star,y)\ge k+1$ and $d_{T'}(s^\star,y)\ge k+1$ both hold.
\item[(ii)]
Let $s^\star$ be any sensor in $\wT_0\cup\wT_1$. Then, for any $y\in V\setminus(\wT_0\cup\wT_1)$, either $d_T(s^\star,y)\ge k+1$, or the path $\cP_T(s^\star,y)$ contains $w_q$.
\item[(iii)]
Let $s^\star$ be any sensor in $\wT_0$. Then, for any $y\in\wT_1$, it holds that $\{w_q,u_{q+1}\}\in V(\cP_T(s^\star,y))$. The same is true if $s^\star \in \wT_1$ and $y \in \wT_0$.
\end{itemize}
\end{claim}
\begin{proof}
The proofs parts (i)--(ii) are completely analogous to those of Claim \ref{claim:nocomm}(i)--(ii), with $s, A_T(s)$ there replaced by $w_q, A^\star_T(w_q)$ here. Part (iii) is immediate by the fact that the only edge connecting vertices of $\wT_0$ and $\wT_1$ in $T$ is $\{w_q,u_{q+1}\}$.
\end{proof}
\subsection{Proof that Transformation C works}\label{sec:C-proof}
\begin{proof}[Proof of Lemma \ref{lem:trCprop}]
We will prove that for any pair of vertices $x,y\in V'\setminus S$ there is a sensor in $S$ that resolves them in $T'=\tr_C(T,S,s_0,s'_0,s_1)$, similarly to the proofs of Lemma \ref{lem:trAprop}(i) and Lemma \ref{lem:trBprop}(i). We will use the notation of Condition \ref{cond:shortest} and Definition \ref{defn:trC}. We will do a case-distinction analysis with respect to the location of $x$ and $y$ in the components $\wT_i$, $i\ge 0$, and in the vertex sets $A^\star_T(w_q)$ and $\{v^\star\}$. The numbering of the cases is consistent with those in the proofs of Lemma \ref{lem:trAprop}(i) and \ref{lem:trBprop}(i).

{\bf Case 1a:} Assume that $x\in\wT_i$ and $y\in\wT_j$ for some $i\ge 2$, $j\ge 0$, $j\ne i$. Then, since $x\notin A^\star_T(w_q)$, there is a sensor $s'\in S$ that measures $x$ such that $\cP_T(s',x)$ does not contain $w_q$. Then, by Claim \ref{claim:nocomm-C}(i)--(ii), $s'\in\wT_i$. Therefore, the edges of $\cP_T(s',x)$ are unchanged in $T'$, so $s'$ still measures $x$ in $T'$. However, it does not measure $y\in\wT_j$ in $T'$ by Claim \ref{claim:nocomm-C}(i). Hence, $s'$ resolves $x$ and $y$ in $T'$.

{\bf Case 1b:} Assume that $x\in\wT_1$ and $y\in\wT_0$. Since $x\notin A^\star_T(w_q)$, there has to exist a sensor $s(x)\in S$ that measures $x$ in $T$ such that $w_q\notin V(\cP_T(s(x),x))$. By Claim \ref{claim:nocomm-C}(i), $s(x)\notin\cup_{i\ge 2}V(\wT_i)$, and $s(x)$ also cannot be in $\wT_0$, since then $w_q\in V(\cP_T(s(x),x))$ would be the case. Hence, $s(x)\in\wT_1$. By the similar reasoning, there has to exist a sensor $s(y)\in S$ such that $s(y)$ measures $y$, and $w_q\notin V(\cP_T(s(y),y))$, and hence $s(y)\in\wT_0$. It follows that the paths $\cP_T(s(x),x)$ and $\cP_T(s(y),y)$ are unchanged by the transformation, hence, $s(x)$ still measures $x$ in $T'$, and $s(y)$ still measures $y$ in $T'$. This implies that either $s(x)$ or $s(y)$ resolves $x,y$ in $T'$ as follows. For an indirect proof assume that neither $s(x)$ nor $s(y)$ resolves $x,y$ in $T'$. Then, this assumption  implies that $s(x)$ measures both $x$ and $y$ in $T'$ with $d_{T'}(s(x),x)=d_{T'}(s(x),y)$, and the same holds for $s(y)$. It then follows that $\typ_{s(x),s(y)}(x)=\typ_{s(x),s(y)}(y)$ in $T'$. But this cannot be the case, as $x,s(x)\in\wT_1$ and $y,s(y)\in\wT_0$ together imply that
\[
\typ_{s(x),s(y)}(x)<\typ_{s(x),s(y)}(v^\star)<\typ_{s(x),s(y)}(y)
\]
in $T'$, finishing the proof.

{\bf Case 2:} Now assume that $x,y\in\wT_i$ for some $i\ge 2$. Let $s'\in S$ be a sensor that resolves $x$ and $y$ in $T$. Then $s'$ has to measure at least one of $x$ and $y$ in $T$, hence $s'\notin\wT_j$ for $j\ge 2$, $j\ne i$ by Claim \ref{claim:nocomm-C}(i). There are two (sub)cases: either $s'\in\wT_i$, or $s'\in \wT_0\cup \wT_1$. First we consider $s'\in\wT_i$. Then the edges of both $\cP_T(s',x)$ and $\cP_T(s',y)$ are all still present in $T'$, and $s'$ resolves $x$ and $y$ in $T'$.

For the other case, we assume that $s'\in\wT_0\cup\wT_1$. We will prove that $s_0$ also resolves $x,y$ in $T$ in this case, and as a result $s_0$ will also resolve them in $T'$. First, we know that $s'$ measures at least one of $x$ and $y$, say it measures $x$. Then, since $x\in \wT_i$ for some $i\ge 2$, by Claim \ref{claim:nocomm-C}(ii), $\cP_T(s',x)$ contains $w_q$. On the path $\cP_T(w_q,x)$ there has to be at least one vertex in $A^\star_T(w_q)$, let the closest one to $x$ be $u$ ($u$ is unique, otherwise there would be a cycle in $T$). Then, since $\wT_i$ is a connected component in $V\setminus A_T^\star(w_q)$, and $y\in\wT_i$, $u$ is also on the path $\cP_T(w_q,y)$. Hence, $V(\cP_T(s',w_q))\subseteq V(\cP_T(s',u))\subseteq V(\cP_T(s',x))\cap V(\cP_T(s',y))$. This implies that $w_q$ is also contained in the path $\cP_T(s',y)$. On the other hand, by the discussion after Condition \ref{cond:shortest}, we have that $d_T(s_0,w_q)\le d_T(s',w_q)$, hence
\[
d_T(s_0,x)\le d_T(s_0,w_q)+d_T(w_q,x)\le d_T(s',w_q)+d_T(w_q,x)=d_T(s',x)\le k,
\]
and thus $s_0$ measures $x$ in $T$. Then, by Claim \ref{claim:nocomm-C}(ii), the path $\cP_T(s_0,x)$ contains $w_q$. Then, by the same reasoning as above (changing $s'$ to $s_0$), we get that $\cP_T(s_0,y)$ also contains $w_q$. Consequently,
\[
d_T(s_0,x)-d_T(s_0,y)=d_T(w_q,x)-d_T(w_q,y)=d_T(s',x)-d_T(s',y),
\]
proving that $s_0$ indeed also resolves $x,y$ in $T$ if $s'$ resolves them in $T$.

Next, we will prove that $s_0$ then also resolves $x,y$ in $T'$. Recall $x_i=\argmin_{v\in\wT_i}d_T(w_q,v)$ from Definition \ref{defn:trC}.
To obtain $T'$, we cut the edges adjacent to $A_T^\star(w_q)$ and replaced them by $\{s_0,x_i\}\in E_3$ (an edge added when creating $T'$). Since every path $\cP_T(s_0,v)$, $v\in\wT_i$, starts with the segment $\cP_T(s_0,x_i)$ in $T$, which we replaced with the single edge $\{s_0, x_i\}$ to obtain $\cP_{T'}(s_0,v)$, the following holds for all $v\in\wT_i$ (for any $i\ge 2$):
\begin{equation}\label{eq:trCeq1}
d_{T'}(s_0,v)=d_T(s_0,v)-d_T(s_0,x_i)+1\le d_T(s_0,v),
\end{equation}
 Hence,
\begin{equation}\label{eq:trCeq2}
d_{T'}(s_0,x)-d_{T'}(s_0,y)=d_{T}(s_0,x)-d_{T}(s_0,y),
\end{equation}
and these distances in $T'$ are no longer than in $T$.
So, if $s_0$ resolved $x$ any $y$ in $T$, then it still resolves them in $T'$. This finishes the proof of Case 2.

{\bf Case 3a:} Next, suppose that $x,y\in\wT_0$, and let $s^\star$ be a sensor that resolves them in $T$. If $s^\star\in\wT_0$, then the paths $\cP_T(s^\star,x)$ and $\cP_T(s^\star,y)$ remain unchanged by the transformation, and thus $s^\star$ still resolves $x,y$ in $T'$. Now assume that $s^\star\notin\wT_0$. By Claim \ref{claim:nocomm-C}(i), $s^\star\in\wT_1$ has to hold. We will prove that in this case either $s_0$ or $s'_0$ will resolve $x,y$ in $T'$. First, since $s^\star\in\wT_1$ and $x,y\in\wT_0$, it has to hold that $w_q\in V(\cP_T(s^\star,x))\cap V(\cP_T(s^\star,y))$ by Claim \ref{claim:nocomm-C}(iii). Since $s^\star$ measures at least one of $x,y$, say it measures $x$, we have
\begin{align}
k\ge d_T(s^\star,x)&=d_T(s^\star,w_q)+d_T(w_q,x)\nonumber\\
&\ge \max\{d_T(s_0,w_q),d_T(s'_0,w_q)\}+d_T(w_q,x)
\ge\max\{d_T(s_0,x),d_T(s'_0,x)\}\nonumber,\label{eq:trC3a1}
\end{align}
where in the first inequality we used the argument after Condition \ref{cond:shortest}. Hence, both $s_0$ and $s'_0$ measure $x$ in $T'$. Now assume indirectly that neither $s_0$ nor $s'_0$ resolves $x,y$ in $T'$, then
\begin{equation}
  \typ_{s_0,s'_0}(x)=\typ_{s_0,s'_0}(y)=:t, \qquad
\hgt_{s_0,s'_0}(x)=\hgt_{s_0,s'_0}(y)=:h\label{eq:trC3a3}  
\end{equation}
in $T'$. Then the same hold in $T$ as all of $s_0,s'_0,x,y\in\wT_0$, hence the paths between them are all unchanged by the transformation. Then \eqref{eq:trC3a3} implies that
\[
d_T(s^\star,x)=d_T(s^\star,w_q)+|q-t|+h=d_T(s^\star,y),
\]
contradicting the fact that $s^\star$ resolved $x,y$ in $T$. This finishes the proof.

{\bf Case 3b:} Now assume that $x,y\in\wT_1$, and let $s^\star$ be a sensor that resolves them in $T$. If $s^\star\in\wT_1$, then the paths $\cP_T(s^\star,x)$ and $\cP_T(s^\star,y)$ remain unchanged by the transformation, and thus $s^\star$ still resolves $x,y$ in $T'$ and we are done. Now assume that $s^\star\notin\wT_1$. By Claim \ref{claim:nocomm-C}(i), $s^\star\in\wT_0$ has to hold. Then, by part (iii) of the same claim, $u_{q+1}\in V(\cP_T(s^\star,x))\cap V(\cP_T(s^\star,y))$. We will show that this implies that in $T'$ $s_0$ will resolve $x,y$. Recall that in $T'$, the length-$2$ path $(s_0, v^\star, u_{q+1})$ connects $s_0$ to $T_1$. Hence for any $v\in \wT_1$,
\begin{align}
d_{T'}(s_0,v)&=d_{T'}(s_0,u_{q+1})+d_{T'}(u_{q+1},v)=2+d_T(u_{q+1},v),\label{eq:trC3b1}
\end{align}
and since $d_T(s^\star, u_{q+1}) \ge 2$ for all $s^\star\in S$,
\begin{align}
k&\ge d_{T}(s^\star,v)=d_{T}(s^\star,u_{q+1})+d_{T}(u_{q+1},v)\ge 2+d_T(u_{q+1},v),\label{eq:trC3b3}
\end{align}
The combination of \eqref{eq:trC3b1} and \eqref{eq:trC3b3} with $v=x $ and $v=y$, respectively, shows that $s_0$ measures both $x$ and $y$ in $T'$. Since $d_T(s^\star,x)\ne d_T(s^\star,y)$, \eqref{eq:trC3b3} applied twice for $v=x$ and $v=y$  shows that $d_T(u_{q+1},x)\ne d_T(u_{q+1},y)$, showing in turn by \eqref{eq:trC3b1} that $d_{T'}(s_0,x)\ne d_{T'}(s_0,y)$. This finishes the proof that $s_0$ resolves $x,y$ in $T'$.

{\bf Case 4:} Assume that $x,y\in A^\star_T(w_q)$. In $T'$ they will then both lie on a single leaf-path emanating from $w_q$. By the remarks after Definition \ref{defn:trC}, $x$ and $y$ are at different distances from $w_q$ in $T'$, and both are measured by both $s_0$ and $s'_0$, implying that both $s_0$ and $s'_0$ resolve them in $T'$.

{\bf Case 5:} Assume that $x\in\wT_i$ for some $i\ge 2$, and $y\in A^\star_T(w_q)$. The proof in this case is exactly the same as in Case 1a.

{\bf Case 6a:} Assume that $x\in\wT_0$ and $y\in A^\star_T(w_q)$. As noted before, $y$ will be measured by both $s_0$ and $s'_0$ in $T'$. Assume that neither of these two sensors resolve $x,y$ in $T'$. This then implies that $d_{T'}(s_0,x)=d_{T'}(s_0,y)\le k$, and the same holds for $s'_0$ in place of $s_0$. It then follows that
\begin{align}
\typ_{s_0,s'_0}(x)=\typ_{s_0,s'_0}(y)=q, \qquad 
\hgt_{s_0,s'_0}(x)=\hgt_{s_0,s'_0}(y)=:h\label{eq:trC6a2}
\end{align}
in $T'$. (Such a scenario can be seen on the right picture of Figure \ref{fig:trC}, with $q=h=1$, with $y=v_1$ and $x$ being the vertex right above $v_1$, measured by $s_6$.)  Since $x\notin A^\star_T(w_q)$, there has to be a sensor $s(x)\in S\setminus \{s_0, s_0'\}$ such that $s(x)$ measures $x$ in $T$, and $w_q\notin V(\cP_T(s(x),x))$. By \eqref{eq:trC6a2} and since $h\ge 1$,  by Claim \ref{claim:nocomm-C}(i),(iii), this implies that $s(x)\in\wT_0$, moreover, $\typ_{s_0,s'_0}(s(x))=q$ in both $T$ and $T'$ (another type would mean that the path $\cP_T(s(x),x)$ passes through $w_q$, since $\mathrm{typ}_{s_0, s_0'}(x)=q=\mathrm{typ}_{s_0, s_0'}(w_q)$, a contradiction with $x\notin A_T^\star(w_q)$). Let the vertex of the path $\cP_T(s(x),x)$ that is closest to $w_q$ be $z$ (on Figure \ref{fig:trC} $z$ and $x$ coincide, but this is not necessarily the case if $h\ge 2$). Then $z\ne w_q$ by $w_q\notin V(\cP_T(s(x),x))$. Notice that the edges of the paths $\cP_T(s(x),x)$ and $\cP_T(s(x),w_q)$ are unchanged by the transformation, as these paths entirely lie in $\wT_0$. Also note that $z\in V(\cP_T(w_q,x))$ and then $d_T(z,x)+d_T(z,w_q)=h$.  With this, we can write the following:
\begin{align*}
d_{T'}(s(x),y)&=d_{T'}(s(x),z)+d_{T'}(z,w_q)+d_{T'}(w_q,y)\\
&=d_T(s(x),z)+d_T(z,w_q)+h\\
&>d_T(s(x),z)+d_T(z,x)=d_{T'}(s(x),x),
\end{align*}
where in the last line we used that $d_T(z,x)<d_T(w_q,x)=h$. This proves that $s(x)$ resolves $x,y$ in $T'$.

{\bf Case 6b:} Assume that $x\in\wT_1$ and $y\in A^\star_T(w_q)$. We will prove that in this case either $s_0$ or $s'_0$ will resolve $x,y$ in $T'$. Assume indirectly that it is not the case. Then \eqref{eq:trC6a2} of Case 6a applies for the same reason, which is a contradiction, since $\typ_{s_0,s'_0}(x)=0$ in $T'$ (as the only the pair of edges $\{s_0,v^\star\}\cup \{v^\star,u_{q+1}\}$ connects $\wT_0$ with $\wT_1$ in $T'$).

{\bf Case 7:} Finally, assume that $x=v^\star\ne y$. Then we prove that either $s_0$ or $s_1$ will resolve $x,y$ in $T'$. Assume indirectly that this does not hold. Since both $s_0$ and $s_1$ measure $v^\star$ in $T'$ (as $d_{T'}(s_1,v^\star)=d_{T}(s_1,w_q)\le k$), this implies that $d_{T'}(s_0,y)=d_{T'}(s_0,v^\star)\le k$, and the same holds for $s_1$ in place of $s_0$. Consequently, $\typ_{s_0,s_1}(y)=\typ_{s_0,s_1}(v^\star)$ in $T'$. But this is impossible, as  $\typ_{s_0,s_1}(v^\star)=1$, and $v^\star$ is the only such type-1 vertex in $T'$, as it has no other neighbors than the two on the path $\cP_{T'}(s_0,s_1)$. This contradiction finishes the proof.
\end{proof}

\section{The size of the optimal tree}\label{sec:OptTree}
Recall $\mathcal T_m^\star$, the set of trees with maximal number of vertices that can be resolved using a sensor set of size $m$.

\begin{defn}
Let $\mathcal T_m^{\star\star}\subseteq\mathcal T_m^\star$ be the set of trees $T$ such that there is a threshold-$k$ resolving set $S(T)$ on $T$ for which $|S(T)|=m$, and for which Condition \ref{cond:trB}(i) holds. In case $S(T)$ is not unique we fix an arbitrary such choice.
\end{defn}

Notice that $\mathcal T_m^{\star\star}\ne\emptyset$, since the application of Lemma \ref{lem:trA} for any $T\in\mathcal T_m^{\star}$ results in a tree $\widehat{T}\in\mathcal T_m^{\star\star}$. Hence, giving the size of any tree in $\mathcal T_m^{\star}$ is equivalent to giving the size of any tree in $\mathcal T_m^{\star\star}$.

\begin{lemma}[Number of sensor paths]\label{lem:NumberOfSensorpaths}
Let $T=(V,E)\in\cT_m^{\star\star}$ and consider the sensor set $S=S(T)$ on it. Then $T$ has $m-1$ sensor paths.
\end{lemma}

\begin{proof} We `renormalise' the tree $T\in \mathcal T_m^{\star\star}$: we contract every sensor path to be a single edge and delete all vertices that are not sensors. This gives us $H  = (V_2, E_2)$ with $V_2 = S$ and $ \{s_1, s_2\} \in E_2$ if $s_1, s_2 \in S$ and there is a sensor path between $s_1$ and $s_2$ in $T$.

We now prove that $H$ is a tree.
$H$ is connected: if there were any $s_1, s_2$ in $H$ with no path between them, then there would also not be a path between $s_1, s_2$ in $T$, which contradicts with $T$ being a tree. Then, assume $H$ has a cycle $(s_1, s_2, \dots, s_n, s_1)$. Then the union of the sensor paths between these consecutive $s_i$'s would form a cycle in $T$, as these sensor paths in $T$ must be disjoint by Lemma \ref{lem:trC} and Corollary \ref{cor:trC}. But $T$ having a cycle contradicts with $T$ being a tree, so $H$ cannot have a cycle. Hence, $H$ is a tree as it is a connected graph without cycles.
As $H$ is a tree on $m$ vertices, it has $m-1$ edges, meaning that there are $m-1$ sensor paths in $T$. 
\end{proof}
Lemma \ref{lem:NumberOfSensorpaths} tells us that $T$ has $m-1$ sensor paths, arranged in a tree-structure $H$. Now we optimize the number of vertices that can be identified by each of these sensor paths. 

\begin{lemma}[Number of vertices on each sensor path]\label{lem:NumberOfVertices}
Consider a tree $T\in \mathcal T_m^{\star\star}$ with the threshold-$k$ resolving set $S=S(T)$ for some $m\ge 1$. The maximal number of vertices in $A_T(s_0, s_1)$ of two neighboring sensors $s_0, s_1$  is $(k^2+k+1)/3$ if $k\equiv 1\pmod 3$ and $(k^2+k)/3$ otherwise.
\end{lemma}

\begin{proof}
By Lemma \ref{lem:trC} and Corollary \ref{cor:trC} we know that all sensor paths in $T$ are disjoint and strong, i.e., they have at most $k+1$ edges. Consider the sensor path between two neighboring sensors $s_0, s_1 \in S$. Using Definition \ref{def:direct-measuring}, disjointness of sensor paths implies that the vertices in $V(\cP_T(s_0, s_1))\setminus \{s_0, s_1\}$ are not directly measured by any sensor in $S\setminus \{s_0, s_1\}$. Hence, each vertex on the sensor path $\cP_T(s_0, s_1)$ belongs to $A_T(s_0, s_1)$.  Recall now the types and heights of vertices from Definition \ref{defn:type-height}.  Using Remark \ref{rem:two-attr}, and the observation before Definition \ref{defn:trB}, all vertices in $A_T(s_0, s_1)$ must have  types  between $1$ and  $d_T(s_0, s_1)-1$ with respect to $s_0, s_1$, and these indeed all belong to $A_T(s_0, s_1)$. Further, again by the observation before Definition \ref{defn:trB},   they all  must have different (type, height) vectors.

Denote by $d:=|V(\cP_T(s_0, s_1))|-2$ the number of vertices between the sensors of the sensor path, so the distance between the sensors is $d+1$. 
By the observation before Definition \ref{defn:trB}, the maximal number of different height values that can belong to type-$i$ vertices is $\min\{k-i, k-(d+1-i)\}+1$, since both $s_0$ and $s_1$ have to measure these vertices, the plus one is because of the vertex with height $0$. 

So the maximum number of vertices in $A_T(s_0, s_1)$ is 
\begin{equation}\label{eq:max-attraction}
    \max_{d \le k+1} (|A_T(s_0, s_1)|)=\max_{d\le k+1} \sum_{i=1}^d\Big( 1+ \min\{k-i, k-(d+1-i)\}\Big).
\end{equation} The inner sum (denote it by $\mathrm{Sum}(d)$) can be simplified, namely if $d$ is even:
\begin{equation} \label{eq1}
\mathrm{Sum}_{e}(d)= d+2\sum_{i=d/2+1}^{d} (k-i) 
= dk - \frac{3d^2}{4}+\frac{d}{2},
\end{equation}
and if $d$ is odd, we have:\\
\begin{equation} \label{eq2}
\begin{split}
\mathrm{Sum}_{o}(d)&=  d+(k-\frac{d+1}{2})+2\sum_{i=1}^{(d-1)/2} (k-(d+1-i))
\\ & = dk - \frac{3d^2}{4}+\frac{d}{2}+\frac{1}{4}.
\end{split}
\end{equation}
Observe that $\mathrm{Sum}_{\Box}(d)$ is a concave parabola of $d$ for both $\Box=o, e$. In both cases the continuous maximizer results in $d = (2k+1)/3$. Because the formulas are quadratic, the maximal integer value of the formula is found by rounding $(2k+1)/3$ to the closest integer.

Here, we distinguish three cases depending on the value of $(k\!\mod 3)$. 

\begin{enumerate}
\item If $k\equiv 0\pmod 3$, then the closest integer to $(2k+1)/3$ is $d=2k/3$. This value is always even, so we use $d_\star = 2k/3$ in \eqref{eq1}, which gives $(k^2+k)/3$.
\item If $k\equiv 1\pmod 3$, then the closest integer to $(2k+1)/3$ is $(2k+1)/3$. This value is always odd, so we substitute $d_\star = (2k+1)/3$ into \eqref{eq2}, which gives $(k^2+k+1)/3$.
\item If $k\equiv 2\pmod 3$, then the closest integer to $(2k+1)/3$ is $(2k+2)/3$. This value is always even, so we substitute $d_\star = (2k+2)/3$ into \eqref{eq1}, which gives $(k^2+k)/3$.
\end{enumerate}
Observe that the optimiser $d_\star\le k+1$ holds in all cases and for all $k$, even if we would drop the restriction of $d\le k+1$ in \eqref{eq:max-attraction}. (This means that in principle one could allow weak sensor paths in the optimisation but one would not gain extra vertices on them.)
\end{proof}

We can now prove Proposition \ref{prop:main}.

\begin{proof}[Proof of Proposition \ref{prop:main}]
By the remark above Lemma \ref{lem:NumberOfSensorpaths} it is sufficient to restrict to $T^\star\in\mathcal{T}_m^{\star\star}$ with the corresponding threshold-$k$ resolving set $S=S(T^\star)$.
Any such $T^\star$ has $m$ sensors that have their attraction in a leaf-path of length $k$ attached to the sensors themselves, accounting for $(k+1)m$ vertices. $T$ also has $m-1$ sensor paths, each carrying the maximal possible size of the attraction of two neighboring sensors, which is  $(k^2+k+1)/3$ vertices if $k\equiv 1\pmod 3$ and $(k^2+k)/3$ otherwise by Lemma \ref{lem:NumberOfVertices}. In total, this means $|T^\star| = (k+1)m+(m-1)(k^2+k+1)/3$ if $k\equiv 1\pmod 3$ and $|T^\star| = (k+1)m+(m-1)(k^2+k)/3$ otherwise.
\end{proof}

See Figure \ref{fig:opt} for two examples of the optimal construction described in the proof.

\begin{figure}[ht]
    \centering
    \begin{subfigure}[b]{0.49\textwidth}
    \centering
    \includegraphics[width=\textwidth]{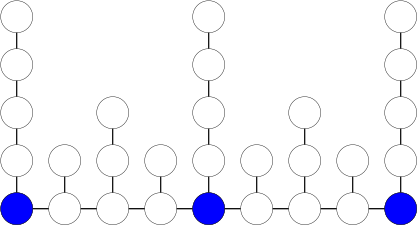}
    \end{subfigure}
    \hfill
    \begin{subfigure}[b]{0.49\textwidth}
    \centering
    \includegraphics[width=\textwidth]{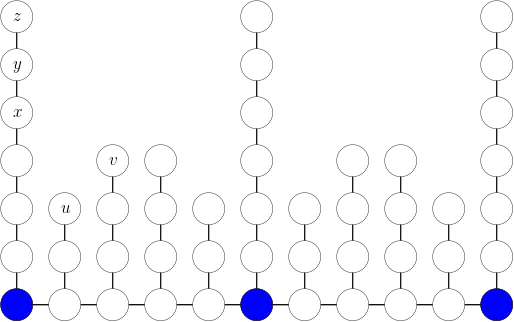}
    \end{subfigure}
    \caption{Two examples for the optimal construction from the proof of Proposition \ref{prop:main}. In both cases, the number of sensors (blue vertices) is $m=3$, while the measuring radius is $k=4$ for the figure on the left, and $k=6$ for the figure on the right. In the case of latter, if we disconnect the vertices $x,y,z$ from their current locations, and instead connect $x$ to $u$, then $y$ to $x$, and $z$ to $v$, then we get another construction with optimal size (that is, a graph in $\cT_3^\star$), but Condition \ref{cond:trB}(i) will be violated by the leftmost sensor, hence this graph will not be in $\cT_3^{\star\star}$.}
    \label{fig:opt}
\end{figure}

\begin{remark}
We only constructed the optimal trees in $\cT_m^{\star\star}$ in our proof. However, there are trees of optimal size that do not satisfy Condition \ref{cond:trB}(i), that is, they belong to $\cT_m^\star\setminus\cT_m^{\star\star}$. An example of how such a tree could be immediately obtained from our constructions is illustrated in Figure \ref{fig:opt}. On the other hand, it follows from our proofs that the repetitive application of Transformation A on any tree in $\cT_m^\star\setminus\cT_m^{\star\star}$ will result in a tree in $\cT_m^{\star\star}$, which follows the construction that we described.
\end{remark}

\section{Improved lower bound based on leaves}\label{sec:general}

Recall the notation  $\cL_v=\{\cP^{(v)}_j\}$ for the collection of leaf-paths starting at a vertex $v$, and their number $L_v=|\cL_v|$, and that the length of a leaf-path $\cP^{(v)}_j$ is denoted by $\ell(\cP^{(v)}_j)=\ell(v,j)$. Furthermore, we will denote the vertices of $\cP^{(v)}_j$ other than $v$ by $x^{(v,j)}_1, x^{(v,j)}_2,\ldots,x^{(v,j)}_{\ell(v,j)}$, in order of increasing distance from $v$. For a generic $\cP\in\cL_v$ we will also denote its vertices other than $v$ in order by $x^{(v)}_1, x^{(v)}_2,\ldots,x^{(v)}_{\ell(\cP)}$.
Recall the support vertices $F_T$ from Definition \ref{def:support} and the upper and lower complexities $\uc(\ell), \lc(\ell)$ from Definition \ref{def:complexity}.

\begin{proof}[Proof of Lemma \ref{lem:LeafPathSensors}]
Let us give some heuristics first: if the length of the leaf-path is at least  $3k+2$, we can place two sensors at distance $k$ and at distance $2k+1$ from the end-vertex (the leaf) of the leaf-path. These two sensors then resolve  the last section containing $3k+2$ vertices. We can then `cut this section off' and iterate the procedure until the length of the remaining leaf-path is strictly shorter than $3k+2$. Then we treat the remaining short leaf-paths together, and one of them will be the special path $P^\star$ that might need one less sensor since it could be measured via a sensor through $v$. To show that this procedure is optimal, we prove by induction. 

The base case is the following: all $P^{(v)}_j\in\cL_v$ have length at most $3k+1$. Let us first assume that this base case holds.

If $P^{(v)}_j\in\cL_v$ has length $\ell(v,j)\in[2k+2,3k+1]$, then the end vertex $x^{(v,j)}_{\ell(v,j)}$ needs to be measured by a sensor $s$ in $\{x^{(v,j)}_{\ell(v,j)-k},x^{(v,j)}_{\ell(v,j)-k+1},\ldots,x^{(v,j)}_{\ell(v,j)}\}$, say $s=x^{(v,j)}_i$. Furthermore, $s$ cannot distinguish between $x^{(v,j)}_{i-1}$ and $x^{(v,j)}_{i+1}$ unless there is another sensor among $x^{(v,j)}_{i-k-1}, x^{(v,j)}_{i-k},\ldots, x^{(v,j)}_{\ell}$ (note that $i-k-1\ge \ell(v,j)-2k-1\ge 1$). Hence, we do need at least two sensors in $V(P_j^{(v)})\setminus\{v\}$, which is exactly (both) $\uc(\ell), \lc(\ell)$ for $\ell\in[2k+2, 3k+1]$.

If $P^{(v)}_j\in\cL_v$ has length $\ell(v,j)\in[k+1,2k+1]$, then the end vertex $x^{(v,j)}_{\ell(v,j)}$ again has to be measured by a sensor in $\{x^{(v,j)}_{\ell(v,j)-k},x^{(v,j)}_{\ell(v,j)-k+1},\ldots,x^{(v,j)}_{\ell(v,j)}\}$ (note that $\ell(v,j)-k\ge 1$), so $V(P^{(v)}_j)\setminus\{v\}$ needs to contain at least one sensor, which gives (both) $\uc(\ell), \lc(\ell)$ for $\ell\in [k+1,2k+1]$.

Now consider all $P^{(v)}_j\in\cL_v$ that have length $\ell(v,j)\in[1,k]$, and call these \textit{short leaf-paths}. In order to distinguish between the vertices $\{x^{(v,j)}_{1}\}_j$ of all short leaf-paths $\{P^{(v)}_j\}_j$, all but at most one of them need to contain a sensor that is not at $v$: this gives $\uc(\ell) $ for all but one short leaf-paths, and gives $\lc(\ell)$ for a single short leaf-path. This finishes the proof for the base case.

Now assume that some $P^{(v)}_j\in\cL_v$ has length $\ell(v,j)\ge 3k+2$. Then, similarly to the first case above, $x^{(v,j)}_{\ell(v,j)}$ can only be measured by a sensor in $x^{(v,j)}_{\ell(v,j)-k},x^{(v,j)}_{\ell(v,j)-k+1},\ldots,x^{(v,j)}_{\ell(v,j)}\}$, say $s=x^{(v,j)}_i$. Furthermore, $s$ cannot distinguish between $x^{(v,j)}_{i-1}$ and $x^{(v,j)}_{i+1}$ unless there is another sensor $s'$ among $x^{(v,j)}_{i-k-1}, x^{(v,j)}_{i-k},\ldots, x^{(v,j)}_{\ell(v,j)}$. Here $i-k-1\ge \ell(v,j)-2k-1$. Then, all the vertices that either $s$ or $s'$ can measure belong to $\{x^{(v,j)}_{\ell(v,j)-3k-1}, x^{(v,j)}_{\ell(v,j)-3k}, \ldots, x^{(v,j)}_{\ell(v,j)}\}$. Hence, the rest of the leaf-paths, that is, $\cup_{k=1}^{L_v}V(P^{(v)}_k)\setminus\{v,x^{(v,j)}_{\ell(v,j)-3k-1}, x^{(v,j)}_{\ell(v,j)-3k}, \ldots, x^{(v,j)}_{\ell(v,j)}\}$ need at least as many sensors as they would need in the graph $T\setminus\{x^{(v,j)}_{\ell(v,j)-3k-1}, x^{(v,j)}_{\ell(v,j)-3k}, \ldots, x^{(v,j)}_{\ell(v,j)}\}$.
Thus, a leaf-path needs an extra $2$ sensors at every multiple of $3k+2$, and this is exactly what both $\uc(\ell)$ and $\lc(\ell)$ express. 
This provides the induction step, and finishes the proof.
(Comment: How many sensors a path needs on its remainder when divided by $3k+2$ is handled by the base case.)
\end{proof}

\begin{proof}[Proof of Theorem \ref{thm:LPLowerBound}]
Assume first that $k\equiv 1\pmod{3}$, and let
\[
B_T:=\left\lceil\frac{3n-3\sum_{v\in F_T}\sum_{j=1}^{L_v}\ell(v,j)+k^2+k+1}{k^2+4k+4}\right\rceil.
\]
For an indirect proof, assume that \eqref{eq:tmdLeafPaths1} does not hold, and in fact there exists a threshold-$k$ resolving set $S^\star$ for $T$ such that
\begin{equation}\label{eq:LPLowerBound1}
|S^\star|\le B_T-1+\sum_{v\in F_T}R(\cL_v)-|F_T|.
\end{equation}
Let
\[V_{LP}:=\bigcup_{v\in F_T}\bigcup_{j=1}^{L_v}V\left(P^{(v)}_j\right)\setminus\{v\},\]
the union of vertices in leaf-paths starting at support vertices, and let
$T':=T\setminus V_{LP}$
be the 'trimmed' version of $T$, when the leaf-paths emanating from the support vertices are removed, but the support vertices still belong to $T'$. Observe that $T'$ is indeed a tree, i.e., connected, since we only removed leaf-paths ending at leafs. By Lemma \ref{lem:LeafPathSensors}, and \eqref{eq:LeafPathSensors},
\[
\left|S^\star\cap V_{LP}\right|\ge \sum_{v\in F_T}R(\cL_v).
\]
Hence, since $T'$ and $V_{LP}$ are on disjoint vertex-sets, by \eqref{eq:LPLowerBound1},
\begin{equation}\label{eq:LPLowerBound2}
|S^\star\cap V(T')|\le B_T-1-|F_T|.
\end{equation}
Consider now the new sensor set $\widetilde S=S^\star\cup F_T$. Since $S^\star$ is a threshold-$k$ resolving set for $T$, so is $\widetilde S$. Moreover, since $F_T\subseteq \widetilde S$, none of the sensors in $\widetilde S\cap V_{LP}$ directly measures any vertex in $T'\setminus \widetilde S$, in the sense of Definition \ref{def:direct-measuring}. This also means that if some sensor $s\in ( V(\mathcal L_v)\setminus \{v\})\cap S^\star$ resolves two vertices $x,y \in T'$, then so does $v\in F_T \cap\widetilde S$.  It then follows that $\widetilde S\cap V(T')$ is a threshold-$k$ resolving set for $T'$. Moreover, since $\widetilde S=S^\star\cup F_T$,  by \eqref{eq:LPLowerBound2},
\[
|\widetilde S\cap V(T')|\le |\widetilde S^\star \cup V(T')| + |F_T| \le  B_T-1.
\]
However, Theorem \ref{thm:main} implies that $\Tmd_k(T')\ge B_T$, as $|V(T')|=n-\sum_{v\in F_T}\sum_{j=1}^{L_v}\ell(v,j)$. This contradiction finishes the proof when $k\equiv 1\pmod{3}$. The proof in the case $k\not\equiv 1\pmod{3}$ is completely analogous.
\end{proof}

\section{Acknowledgements}
We are grateful to Gergely \'Odor for providing us inspiration for this problem, and for pointing out numerous valuable references.

\newpage
\bibliographystyle{abbrv}
\bibliography{main}
\end{document}